\documentclass{article}
\usepackage{authblk}
\usepackage{amsmath}
\usepackage{amsfonts,amssymb}
\usepackage{mathrsfs}
\usepackage{ntheorem}
\usepackage{indentfirst}
\usepackage{cite}
\usepackage{geometry}
\usepackage{graphicx}
\usepackage{subfigure}
\usepackage{epstopdf}
\usepackage{authblk}
\usepackage{float}
\usepackage{color}

\newtheorem{theorem}{Theorem}
\newtheorem{lemma}{Lemma}
\newtheorem{remark}{Remark}
\newtheorem{proposition}{Proposition}
\newtheorem*{acknowledgement}[theorem]{Acknowledgement}

\newenvironment{proof}[1][Proof]{\noindent\textbf{#1.} }{\ \rule{0.5em}{0.5em}}
\begin{document}
\title{\textbf{Linear-Quadratic Mean Field Games of Controls with Non-Monotone Data}}
\author{Min Li\footnote{School of Mathematics, Shandong
University, Jinan 250100, China; Email: {\tt liminmath@outlook.com}},\quad Chenchen Mou\footnote{Corresponding author. Department of Mathematics, City University of Hong Kong, Hong Kong, China; Email: {\tt chencmou@cityu.edu.hk}}, ~~Zhen Wu\footnote{ School of Mathematics, Shandong University, Jinan 250100, China; Email: {\tt wuzhen@sdu.edu.cn}}\quad and Chao Zhou\footnote{Department of Mathematics and Risk Management Institute, National University of Singapore, Singapore; Email: {\tt matzc@nus.edu.sg}}}
\date{}
\maketitle

\begin{abstract}
In this paper, we study a class of linear-quadratic (LQ) mean field games of controls with common noises and their corresponding $N$-player games. The theory of mean field game of controls considers a class of mean field games where the interaction is via the joint law of both the state and control.  By the stochastic maximum principle, we first analyze the limiting behavior of the representative player and obtain his/her optimal control in a feedback form with the given distributional flow of the population and its control. The mean field equilibrium is determined by the Nash certainty equivalence (NCE) system.  Thanks to the common noise, we do not require any monotonicity conditions for the solvability of the NCE system. We also study the master equation arising from the LQ mean field game of controls, which is a finite-dimensional second-order parabolic equation. It can be shown that the master equation admits a unique classical solution over an arbitrary time horizon without any monotonicity conditions. Beyond that, we can solve the $N$-player game directly by further assuming the non-degeneracy of the  idiosyncratic noises. As byproducts, we prove the quantitative convergence results from the $N$-player game to the mean field game and the propagation of chaos property for the related optimal trajectories.  \end{abstract}



\section{Introduction}
 The theory of mean field games is devoted to studying strategy decision making in large populations where the individuals interact through certain mean field quantities. The introduction to mean field games were the pioneering works of Lasry-Lions \cite{LasryLions2007} and Huang-Malham\'e-Caines \cite{Huang2006}. See, e.g. Lions \cite{Lions2007}, Cardaliaguet \cite{Cardaliaguet2010}, Bensoussan-Frehse-Yam \cite{BensoussanFrehseYam2013}, Carmona-Delarue \cite{CarmonaDelarue20181,CarmonaDelarue20182}, Cardaliaguet-Porretta \cite{CardaliaguetPorretta2020} and the reference therein for more details on profound theoretical results and the broad applications of mean field games. Generally, mean field games are to solve their associated $N$-player games, of which the dimension increases as the number of players increases. Therefore, it is really difficult to numerically solve the $N$-player Nash equilibria directly via its optimality conditions. On the contrary, the mean field game approach provides an effective and promising solution since its dimension is fixed.  
 
Mean field games of controls are a class of mean field games where the interaction is via the joint law of both the state and control, which are called the extended mean field games at the early stage. See, e.g. Gomes-Voskanyan \cite{GomesVoskanyan2016}, Gomes-Patrizi-Voskanyan\cite{GomesPatriziVoskanyan2014}, Carmona-Lacker \cite{CarmonaLacker2015}, Carmona-Delarue \cite{CarmonaDelarue20181},  Cardaliaguet-Lehalle \cite{CardaliaguetLehalle2018}, Achdou-Kobeiss\cite{AchdouKobeiss2021},  Kobeissi \cite{Kobeissi2021,Kobeissi2020},  Bonnans-Hadikhanloo-Pfeiffer \cite{BonnansHadikhanlooPfeiffer2021}, Graber-Mayorga \cite{GraberMayorga}.

The master equation was introduced by Lions \cite{Lions2007} in his lectures on the mean field game and its applications at \emph{Coll\`ege de France}, which is an infinite-dimensional partial differential equation to characterize the value of the mean field game. The well-posedness of the master equation has always been one of great focuses in the mean field game community. When the data are sufficiently smooth, the master equation usually admits a classical solution on a small time horizon, see, e.g. Gangbo-\'Swi\c{e}ch \cite{GangboSwiech2015}, Carmona-Delarue \cite{CarmonaDelarue20182},  Bensoussan-Yam \cite{BensoussanYam2019},  Cardaliaguet-Cirant-Porretta\cite{CardaliaguetCirantPorretta2020}. By contrast, the global well-posedness of the classical solutions to the master equation is more challenging, see, e.g.
Buckdahn-Li-Peng-Rainer \cite{BuckdahnLiPengRainer2017},   Chassagneux-Crisan-Delarue \cite{ChassagneuxCrisanDelarue2014}, Cardaliaguet-Delarue-Lasry-Lions \cite{CardaliaguetDelarueLasryLions2019}, Carmona-Delarue \cite{CarmonaDelarue20182}, Bayraktar-Cohen \cite{BayraktarCohen2018}, Gangbo-M\'esz\'aros \cite{GangboMeszaros2020}, Bensoussan-Graber-Yam \cite{BensoussanGraberYam2019,BensoussanGraberYam2020},  Bertucci-Lasry-Lions \cite{BertucciLasryLions2019}, Gangbo-M\'esz\'aros-Mou-Zhang \cite{GangboMou2021}, Mou-Zhang \cite{MouZhang2022}. There are also some studies on the global well-posedness of the weak solutions to the master equation, see e.g. Mou-Zhang \cite{MouZhang2019}, Bertucci \cite{Bertucci2020,Bertucci2021}, Cardaliaguet-Souganidis \cite{CardaliaguetSouganidis2021}, Cecchin-Delarue \cite{CecchinDelarue20221}.

In the literature, the monotonicity condition plays a key role in the global well-posedness of the master equation as well as the uniqueness of mean field equilibria and solutions to mean field game systems.  For example, \cite{ChassagneuxCrisanDelarue2014,CardaliaguetDelarueLasryLions2019,CarmonaDelarue20182,BayraktarCohen2018,BertucciLasryLions2019, MouZhang2019} are based on the well-known Lasry-Lions monotonicity condition. Another monotonicity condition is called the displacement monotonicity condition, see, e.g. \cite{BensoussanGraberYam2019, BensoussanGraberYam2020,GangboMeszaros2020,GangboMou2021}. Noting that both monotonicity conditions are always taken in a fixed direction. \cite{MouZhang2022} proposed a new type of monotonicity condition in the opposite direction, which is called the anti-monotonicity condition. For mean field games without any monotonicity conditions, they may have multiple equilibria, see e.g. Tchuendom \cite{Tchuendom2018}, Bardi-Fischer \cite{BardiFischer2019}, Delarue-Tchuendom \cite{DelarueTchuendom2020}, Bayraktar-Zhang \cite{BayraktarZhang2020}, Cecchin-Delarue\cite{CecchinDelarue2022} for the case of nonlocal coupling; Cirant\cite{Cirant2019}, Cirant-Verzini \cite{CirantVerzini2017}, Cirant-Tonon\cite{CirantTonon2019} for the case of local coupling. Recently, Cecchin-Delarue \cite{CecchinDelarue20221} studies the master equations for potential mean field games with non-monotone data.


One of the main goals of our paper is to study a class of LQ mean field games of controls with common noises and their corresponding master equations. In particular, we do not require any monotonicity conditions to establish the global well-posedness theory. 
In the literature, the master equation usually includes a probability measure as
one of the state variables thus becomes a parabolic equation on the Wasserstein space, which is an infinite-dimensional PDE. In this manuscript, we consider a class of LQ mean field games of controls where the mean field interaction is only through the expectations of the state and control. Therefore, the master equation reduces to be a finite dimensional second-order parabolic equation. The first main result of our manuscript is the global well-posedness of classical solutions to the master equation. To our best knowledge, this is the first global well-posedness result for the master equation in the literature for mean field games of controls. There is a recent progress on (infinite dimensional) master equations for mean field games of controls in \cite{MouZhang20221}, where the propagation of monotonicities along classical solutions of master equations has been proved. To establish the well-posedness result, we follow the following steps. We first use the stochastic maximum principle to obtain the optimal control in a feedback form with the given joint expectation of the state and the control. Then we characterize the mean field equilibrium by the so-called Nash certainty equivalence (NCE) principle, which was introduced in Huang-Malham\'e-Caines \cite{Huang2007}. 
The construction of the NCE system is slightly more difficult than that in the case of the standard mean field game because of the presence of the expectation of the control. We overcome the difficulty by establishing the connection between the expectation of the control and the expectation of the state using the inverse function theorem. We then follow the idea in Tchuendom \cite{Tchuendom2018} and Delarue-Tchuendom \cite{DelarueTchuendom2020} to show the global well-posedness of the NCE system through the non-degeneracy of the common noise without any monotonicity conditions. This is the main reason that we can get rid of any monotonicity conditions.  With the help of the NCE system and the Riccati equation, we are able to show the global well-posedness of the stochastic Hamiltonian system, i.e. the forward-backward stochastic differential equation (FBSDE) derived from the stochastic maximum principle. We can further show that the decoupling field of the stochastic Hamiltonian system has nice regularity property to solve the vectorial master equation. Finally, we use the solution to the vectorial master equation to establish the existence of a solution to the master equation. The uniqueness of solutions to the master equation is a byproduct of the global well-posedness of the stochastic Hamiltonian system.

As an important application of our well-posedness results, we prove the convergence of the $N$-player games and the related propagation of chaos property. There is a broad literature on the convergence from the $N$-player game to the mean field game. See, e.g. \cite{CardaliaguetDelarueLasryLions2019,CarmonaDelarue20182,DelarueTchuendom2020,MouZhang2019}, Delarue-Lacker-Ramanan\cite{DelarueLackerRamanan2019,DelarueLackerRamanan2020}, Lacker \cite{Lacker2016,Lacker2017,Lacker2018,Lacker2020}, Nutz-San Martin-Tan \cite{NutzSanMartinTan2020}, Iseri-Zhang \cite{IseriZhang2021}, Lacker-Flem \cite{Lacker2021}, Djete \cite{Djete2021} for more details.   Unlike most of the existing literature, we do not consider the quantitative convergence problems for the closed-loop Nash equilibria by using the classical solution to the master equation. This is because it is a difficult task to show the global well-posedness of the Nash system, derived from the optimality condition for the closed-loop Nash equilibria, with quadratic growth data. In fact, to the best of our knowledge, such global well-posedness result remains open. Instead, we aim to show quantitative convergence results from the $N$-player game to the mean field game and the propagation of chaos property for the related optimal trajectories corresponding to the open-loop Nash equilibria. We emphasize that, unlike the mean field game, open-loop Nash equilibria are not closed-loop Nash equilibria although they are of closed-loop form.  The key to proving the convergence of the $N$-player games is to use the fact that the finite-dimensional projections of the decoupling field of the NCE system can approximate the decoupling field of the corresponding NCE system for the $N$-player game as $N$ goes to the infinity. We note that the well-posedness of the NCE system for the $N$-player game further needs the non-degeneracy of the idiosyncratic noises. We then follow \cite{CardaliaguetDelarueLasryLions2019} to show the remaining convergence results and the propagation of chaos property for optimal trajectories.

The rest of the paper is organized as follows.  In section 2, we give some necessary preliminaries and formulate our problem.
In section 3, we introduce a class of LQ mean field games of controls with common noises and establish the global well-posedness of the NCE system and the master equation with non-monotone data.  section 4 is devoted to solving the $N$-player games directly. In section 5, we show the convergence of the $N$-player games and the propagation of chaos property for optimal trajectories.

 \section{The setting}
\subsection{ LQ $N$-player game of controls}\label{section:LP}
 
 In this subsection, we consider a class of linear-quadratic (LQ) $N$-player game of controls. 
 
To reveal the nature of the problem, we let our state space be $\mathbb R$ throughout the paper. The results remain valid for higher dimensions. For a given $T>0$, let $(\widetilde\Omega,\widetilde{\mathcal{F}},\widetilde{\mathbb F},\widetilde{\mathbb{P}})$ be a complete filtered probability space on which can support $N+1$ independent one-dimensional Brownian motions: $W^i,~1\leq i\leq N,$ and $W^0$. Here $W^i$ denotes the  idiosyncratic noise for the $i$th player and $W^0$ denotes the common noise for all the players. Let $\widetilde{\mathbb F}:=\{\widetilde{\mathcal{F}}_t\}_{0\leq t\leq T}$ where $\widetilde{\mathcal{F}}_t:=\left(\vee_{i=1}^N \mathcal{F}^{W^i}_t\right)\vee \mathcal{F}_t^{W^0}\vee \widetilde{\mathcal{F}}_0$ and let $\widetilde{\mathbb P}$ has no atom in $\widetilde{\mathcal{F}}_0$. 

We then introduce the following spaces: for any sub-filtration $\mathbb G$ of $\mathbb F$, we denote $m\in L_{\mathbb G}^2([0,T];\mathbb{R})$ if $m: \Omega \times [0,T]\rightarrow \mathbb{R}$ is an $\mathbb G$-adapted process such that $\mathbb{E}\big[\int_0^T|m_t|^2dt\big]<\infty$; for any sub-$\sigma$-field $\mathcal{G}\subset\mathcal{F}$, we denote $\xi\in L^2_{\mathcal{G}}$ if $\xi:\Omega\rightarrow \mathbb{R}$ is a $\mathcal{G}$-measurable random variable such that $\mathbb{E}[|\xi|^2]<\infty$.
 

 For $1\leq i \leq N$, let $\xi^i\in L_{\widetilde{\mathcal{F}}_0}^2$ be independent and identically distribution (i.i.d.) random variables. Let us use $x^i$ and $\alpha^i$ to represent the state and the control processes of the $i$th player, respectively. Denote $\boldsymbol{x}:=(x^1,\ldots,x^N)$ and $\boldsymbol{\alpha}:=(\alpha^1,\ldots,\alpha^N)$. Suppose that the state $x^i$ of the $i$th player is given by
 \begin{equation}
 \left\{
 \begin{aligned}\label{state}
dx^i_t=&~[Ax^i_t+B\alpha^i_t+f(\nu^{N,i}_{\boldsymbol{x}_t})+b(\mu^{N,i}_{\boldsymbol{\alpha}_t})]dt+\sigma dW^i_t+\sigma_0dW^0_t,\\
 x^i_0=&~\xi^i,
 \end{aligned}
 \right.
 \end{equation}
where $\alpha^i\in \widetilde{\mathcal{U}}_{ad}(0,T)=\{\alpha\,|\,\alpha\in L^2_{\widetilde{\mathbb{F}}}([0,T];\mathbb{R}) \}$ and the interactions among players are via the average of all other players' states and controls
 \begin{equation}\label{eq:nuNimuNi}
 \mu^{N,i}_{\boldsymbol{\alpha}}:=\frac{1}{N-1}\underset{j\neq i}{\sum} \alpha^j, \quad  \nu^{N,i}_{\boldsymbol{x}}:=\frac{1}{N-1}\underset{j\neq i}{\sum} x^j.\end{equation}
 
 The cost functional of $i$th player is assumed to be 
 \begin{equation}\label{cost}
\mathcal{J}^i(\alpha^i,\boldsymbol{\alpha}^{-i}):=\frac{1}{2}\mathbb{E}\Big\{\int_0^T\Big[Q\left(x^i_t+l(\nu^{N,i}_{\boldsymbol{x}_t})\right)^2+R\left(\alpha^i_t+h(\mu^{N,i}_{\boldsymbol{\alpha}_t})\right)^2\Big]dt+G\left(x^i_T+g(\nu^{N,i}_{\boldsymbol{x}_T})\right)^2\Big\},
\end{equation}
where $\boldsymbol{\alpha}^{-i}=(\alpha^1,\ldots,\alpha^{i-1},\alpha^{i+1},\ldots,\alpha^N)$.

Then the major problem of the above $N$-player game is to find the Nash equilibrium (NE) $\boldsymbol{\alpha}^*$.\\
\textbf{Problem (NP)}
Find a strategy profile $\boldsymbol{\alpha}^*=(\alpha^{*,1},\ldots,\alpha^{*,N})$ where $\alpha^{*,i}\in \widetilde{\mathcal{U}}_{ad}(0,T)$, $1\leq i \leq N$, such that 
\begin{equation*}
\mathcal{J}^i(\alpha^{*,i},\boldsymbol{\alpha}^{*,-i})=\underset{\alpha^i\in \widetilde{\mathcal{U}}_{ad}(0,T)}{\inf}\mathcal{J}^i(\alpha^i,\boldsymbol{\alpha}^{*,-i}).
\end{equation*}

\subsection{LQ mean field game of controls
\label{sect-MFG}}
In this subsection, we introduce the following LQ mean field game of controls.  The introduction of such mean field game of controls is for the study of the above LQ $N$-player game of controls. We shall prove a quantitative convergence result via the master equation from the $N$-player game to the mean field game to show how well the mean field game can approximate the $N$-player game. Note that it is really hard to find Nash equilibria for the $N$-player games numerically since the dimension increases as the $N$ increases. However, their corresponding mean field games have a fixed dimension.

Let $(\Omega,\mathcal{F},\mathbb F,\mathbb P)$ be a complete filtered probability space on which can support two independent one-dimensional Brownian motions: $W$ and $W^0$. Here $W$ denotes the idiosyncratic noise and $W^0$ denotes the common noise.  We let $\mathbb F:=\{\mathcal{F}_t\}_{t\in[0,T]}$, where $\mathcal{F}_t:=\mathcal{F}^{W}_t\vee\mathcal{F}^{W^0}_t \vee\mathcal{F}_0$, and let $\mathbb P$ have no atom in $\mathcal{F}_0$ so it can support any measure on $\mathbb R$ with a finite second moment. We denote $\mathbb F^0:=\{\mathcal{F}_t^{W^0}\}_{t\in [0,T]}$.

 For any given $(\mu,\nu)\in L^2_{\mathbb F^0}([0,T];\mathbb R)\times L^2_{\mathbb F^0}([0,T];\mathbb R)$ and $\xi\in L_{\mathcal{F}_0}^2$ with $\nu_0=\mathbb E[\xi]$, the dynamics of the representative player is given by
  \begin{equation}\label{lstate}
 \left\{
 \begin{aligned}dx_t^{\xi,\alpha}=&~[Ax_t^{\xi,\alpha}+B\alpha_t+f(\nu_t)+b(\mu_t)]dt+\sigma dW_t+\sigma_0dW^0_t,\\
 x_0^{\xi,\alpha}=&~\xi,
 \end{aligned}
 \right.
 \end{equation}
 where the control process $\alpha\in \mathcal{U}_{ad}(0,T):=\{\alpha\,|\,\alpha\in L^2_{\mathbb{F}}([0,T];\mathbb{R})\}$, the constant coefficients $A,B\in\mathbb R,\sigma\geq 0,\sigma_0>0$ and the functions $f,b:\mathbb R\mapsto\mathbb R$ are bounded and uniformly Lipschitz continuous. The cost functional of the representative player is 
 \begin{equation}\label{lcost}
\mathcal{J}(\mu,\nu;\alpha):=\frac{1}{2}\mathbb{E}\Big\{\int_0^T\Big[Q\left(x_t^{\xi,\alpha}+l(\nu_t)\right)^2+R\left(\alpha_t+h(\mu_t)\right)^2\Big]dt+G\left(x_T^{\xi,\alpha}+g(\nu_T)\right)^2\Big\},
\end{equation}
where the constant coefficients $Q\geq 0, G\geq 0$ and $R>0$ and the functions $l,h,g:\mathbb R\mapsto\mathbb R$ are bounded and uniformly Lipschitz continuous.

We consider the following minimization problem:
\begin{equation}\label{minimization}
\mathcal{V}(\mu,\nu;\alpha^{*,\xi})=\underset{\alpha\in\mathcal{U}_{ad}(0,T)}{\inf} \mathcal{J(}\mu,\nu;\alpha).
\end{equation}

Then the central problem of the above mean field game of controls is to find the following mean field equilibrium (MFE) $\alpha^{*,\xi}$ and its corresponding stochastic measure flow $(\mu^{*,\xi},\nu^{*,\xi})$. \\
\textbf{Problem (MF)} Find $\alpha^{*,\xi}\in\mathcal{U}_{ad}(0,T)$ and $(\mu^{*,\xi},\nu^{*,\xi})\in L^2_{\mathbb F^0}([0,T];\mathbb R)\times L^2_{\mathbb F^0}([0,T];\mathbb R)$ such that 
\begin{equation*}
\mathcal{V}(\mu^{*,\xi},\nu^{*,\xi};\alpha^{*,\xi})=\underset{\alpha\in\mathcal{U}_{ad}(0,T)}{\inf}\mathcal{J}(\mu^{*,\xi},\nu^{*,\xi};\alpha),
\end{equation*}
where $(\mu^{*,\xi},\nu^{*,\xi})$ satisfies the consistency condition
\begin{equation}\label{eq:consistency}
\mu_t^{*,\xi}=\mathbb E[\alpha^{*,\xi}_t|\mathcal{F}_t^{W^0}]\quad\text{and}\quad \nu_t^{*,\xi}=\mathbb E[x_t^{\xi,\alpha^{*,\xi}}|\mathcal{F}_t^{W^0}].
\end{equation}


\subsection{Assumptions}

In the last subsection, we collect all the technical assumptions on data.

\textbf{Assumption (A)}\\
(i) Assume that $A,B,\sigma\geq 0, \sigma_0>0, Q\geq 0, G\geq 0, R>0$ are constant coefficients;\\
(ii) Assume that $f(\cdot),b(\cdot),l(\cdot),h(\cdot),g(\cdot):\mathbb R\to\mathbb R$ are bounded and uniformly Lipschitz continuous functions.

\textbf{Assumption (B)} There exists some $\varepsilon_0>0$ such that $|1+h^{\prime}(\cdot)|\geq\varepsilon_0$.

and 

\textbf{Assumption (C)} Assume that $f,b,l,h,g\in C^2(\mathbb R)$ are bounded functions, and moreover they have bounded 1st and 2nd order derivatives.

\section{Master equation}
\label{sect-Introduction}

\subsection{Stochastic maximum principle}
\label{subsect-SMP}

To solve the Problem (MF), we first solve the minimization problem \eqref{minimization} with the given stochastic measure flow $(\mu,\nu)\in L^2_{\mathbb F^0}([0,T];\mathbb R)\times L^2_{\mathbb F^0}([0,T];\mathbb R)$. 

Let us introduce the following stochastic Hamiltonian system: for any given $\xi\in L^2_{\mathcal{F}_0}$ and $(\mu,\nu)\in L_{\mathbb F^0}^2([0,T];\mathbb R)\times L_{\mathbb F^0}^2([0,T];\mathbb R)$
\begin{equation}\label{H}
\left\{
\begin{aligned}
dx_t^{\xi}=&~[Ax_t^{\xi}-B^2R^{-1}y_t^{\xi}-Bh(\mu_t)+f(\nu_t)+b(\mu_t)]dt\\
&+\sigma dW_t+\sigma_0dW^0_t,\\
dy_t^{\xi}=&-[Ay_t^{\xi}+Qx_t^{\xi}+Ql(\nu_t)]dt+z_t^{\xi}dW_t+z^{0,\xi}_tdW^0_t,\\
x_0^{\xi}=&~\xi,~y_T^{\xi}=G(x_T^\xi+g(\nu_T)),
\end{aligned}
\right.
\end{equation}
which is derived from the minimization problem \eqref{minimization} by using the stochastic maximum principle. Moreover, the corresponding optimal control process is given by
\begin{equation}\label{eq:optalpha}
\alpha_t^{\xi}=-R^{-1}By_t^{\xi}-h(\mu_t).
\end{equation}
We note that the optimal control process $\alpha^{\xi}$ given above is presented in open-loop form. Thanks to the LQ setting, we shall be able to apply the decoupling method to show that such optimal control $\alpha^\xi$ is in fact in a feedback form. To find the feedback form, we introduce the Riccati equation
\begin{equation}\label{P}
\left\{
\begin{aligned}
&\dot{P}_t+2AP_t-B^2R^{-1}P^2
_t+Q=0,\\
&P_T=G,
\end{aligned}
\right.
\end{equation}
and the backward stochastic differential equation (BSDE)
\begin{equation}\label{varphi}
\left\{
\begin{aligned}
d\varphi_t^\xi=&-[(A-B^2R^{-1}P_t)\varphi_t^\xi+P_tf(\nu_t)+P_tb(\mu_t)\\&+Ql(\nu_t)-P_tBh(\mu_t)]dt+\Lambda_t^{0,\xi}dW^0_t,\\
\varphi_T^\xi=&~Gg(\nu_T).
\end{aligned}
\right.
\end{equation}
We note that the above BSDE \eqref{varphi} is driven by the common noise $W^0$ only. We shall prove in the following theorem that under Assumption (A) all the above equations \eqref{H}, \eqref{P}, \eqref{varphi} are well-posed.

\begin{theorem}\label{closeloop}
Suppose that Assumption (A) holds. Let $\xi\in L^2_{\mathcal{F}_0}$ and $(\mu,\nu) \in L^2_{\mathbb{F}^0}([0,T];\mathbb{R})\times L^2_{\mathbb{F}^0}([0,T];\mathbb{R})$.\\
(i) The Riccati equation \eqref{P} and the BSDE \eqref{varphi} admit a unique bounded solution $P$ and a unique solution $(\varphi^\xi,\Lambda^{0,\xi})$ on $[0,T]$, respectively;\\
(ii) The optimal control for the minimization problem \eqref{minimization} is given by the feedback form 
\begin{equation}\label{alpha}
\alpha^\xi_t:=-R^{-1}B(P_tx^\xi_t+\varphi_t^\xi)-h(\mu_t)
\end{equation}
 where its corresponding optimal trajectory follows
\begin{equation}\label{fstate}
 \left\{
 \begin{aligned}
 dx_t^{\xi}=&~[(A-B^2R^{-1}P_t)x_t^{\xi}-B^2R^{-1}\varphi_t^\xi-Bh(\mu_t)\\&+f(\nu_t)+b(\mu_t)]dt+\sigma dW_t+\sigma_0dW^0_t,\\
 x_0^{\xi}=&~\xi;
 \end{aligned}
 \right.
 \end{equation}
 (iii) Given $x^\xi$ in (ii), the BSDE in the stochastic Hamiltonian system \eqref{H} admits a unique solution $(y^\xi,z^\xi,z^{0,\xi})$ on $[0,T]$ where $y_t^\xi=P_tx_t^\xi+\varphi_t^\xi$, $z_t^\xi=P_t\sigma$ and $z_t^{0,\xi}=P_t\sigma_0+\Lambda^{0,\xi}_t$. Therefore, $(x^\xi,y^\xi,z^\xi,z^{0,\xi})$ is the unique strong solution to the stochastic Hamiltonian system \eqref{H}.
\end{theorem}
\begin{proof}
(i) The ordinary differential equation \eqref{P} is a standard Riccati equation and thus it admits a unique bounded solution $P$. We then note that the BSDE \eqref{varphi} is a linear BSDE with bounded coefficients. By the standard BSDE theory (see e.g. \cite{Zhang2017}), the BSDE \eqref{varphi} has a unique solution $(\varphi^\xi,\Lambda^{0,\xi})$.

(ii)(iii) We show the proofs of (ii) and (iii) together. Given $P$ and $\varphi^\xi$ in (i), we define the following feedback control:
\begin{equation}\label{eq:alpharev}
\alpha_t(x):=-R^{-1}B(P_tx+\varphi_t^\xi)-h(\mu_t).
\end{equation}
Plugging the above $\alpha$ into \eqref{lstate}, we obtain the SDE \eqref{fstate}. Since $P$ is bounded and $R>0$, the SDE \eqref{fstate} is well-posed on $[0,T]$ and denote its solution by $x^\xi$. Define
\[y_t^{\xi}:=P_tx_t^{\xi}+\varphi_t^\xi,\quad z_t^\xi:=P_t\sigma,\quad z_t^{0,\xi}:=P_t\sigma_0+\Lambda_t^{0,\xi}.\]
Note that $y_T^\xi=P_Tx_T^\xi+\varphi_T^\xi=G(x_T^\xi+g(\nu_T))$. Applying the It\^{o}'s formula to $y_t^{\xi}$, it yields
 \begin{equation*}
 \begin{aligned}
 dy_t^{\xi}=&[-2AP_t+B^2R^{-1}P_t^2-Q]x_t^{\xi}dt\\
 &+P_t[Ax_t^{\xi}-B^2R^{-1}y_t^{\xi}-Bh(\mu_t)+f(\nu_t)+b(\mu_t)]dt+P_t\sigma dW_t+P_t\sigma_0dW^0_t\\
 &-[(A-B^2R^{-1}P_t)\varphi_t+P_tf(\nu_t)+P_tb(\mu_t)+Ql(\nu_t)-P_tBh(\mu_t)]dt+\Lambda_t^{0,\xi}dW_t^0\\
 =&[-(A-B^2R^{-1}P_t)y_t^\xi-Qx_t^{\xi}-P_tB^2R^{-1}y_t^{\xi}-Ql(\nu_t)]dt\\
 &+P_t\sigma dW_t+(P_t\sigma_0+\Lambda_t^{0,\xi})dW_t^0\\
 =&-[Ay_t^\xi+Qx_t^\xi+Ql(\nu_t)]dt+z_t^{\xi}dW_t+z_t^{0,\xi}dW_t^0.
 \end{aligned}
 \end{equation*}
By the standard BSDE theory, we know that $(y^{\xi},z^\xi,z^{0,\xi})$ is the unique solution to the BSDE in the stochastic Hamiltonian system \eqref{H} on $[0,T]$. Applying $y_t^{\xi}=P_tx_t^{\xi}+\varphi_t^\xi$, we note that $x^\xi$ satisfies the forward stochastic differential equation in \eqref{H} on $[0,T]$. Therefore, we verify that $(x^{\xi},y^{\xi},z^\xi,z^{0,\xi})$ is a strong solution to the stochastic Hamiltonian system \eqref{H}. The uniqueness of the strong solutions to \eqref{H} on $[0,T]$ follows from the standard local well-posedness theory of FBSDEs.

By \eqref{eq:optalpha}, we have $\alpha_t^\xi=-R^{-1}B(P_tx^\xi_t+\varphi_t^\xi)-h(\mu_t)$ is the optimal control for the minimization problem \eqref{minimization}.\end{proof}

\subsection{NCE system}
\label{sebsect-NCE}

Until now, all the results derived in subsection \ref{subsect-SMP} are based on given the stochastic measure flow $(\mu,\nu)\in L^2_{\mathbb F^0}([0,T];\mathbb R)\times L^2_{\mathbb F^0}([0,T];\mathbb R)$. 

In this subsection, we aim to characterize $(\mu^{*,\xi},\nu^{*,\xi})$ which satisfies the consistency condition \eqref{eq:consistency}. By the consistency condition \eqref{eq:consistency} and Theorem \ref{closeloop}, we have $(\mu^{*,\xi},\nu^{*,\xi})$ such that 
\begin{equation*}
\mu_t^{*,\xi}=\mathbb E\Big[\alpha_t^{*,\xi}|\mathcal{F}_t^{W^0}\Big] \quad\text{and}\quad \nu_t^{*,\xi}=\mathbb E\Big[x_t^{*,\xi}|\mathcal{F}_t^{W^0}\Big]
\end{equation*}
where $x^{*,\xi}$ is the strong solution to the SDE
\begin{equation}\label{fstate*}
 \left\{
 \begin{aligned}
 dx_t^{*,\xi}=&~[(A-B^2R^{-1}P_t)x_t^{*,\xi}-B^2R^{-1}\varphi_t^{*,\xi}-Bh(\mu_t^{*,\xi})\\&+f(\nu_t^{*,\xi})+b(\mu_t^{*,\xi})]dt+\sigma dW_t+\sigma_0dW^0_t,\\
 x_0^{*,\xi}=&~\xi,
 \end{aligned}
 \right.
 \end{equation}
$\varphi^{*,\xi}$ is the solution to the BSDE
\begin{equation}\label{varphi*}
\left\{
\begin{aligned}
d\varphi_t^{*,\xi}=&-[(A-B^2R^{-1}P_t)\varphi_t^{*,\xi}+P_tf(\nu_t^{*,\xi})+P_tb(\mu_t^{*,\xi})\\&+Ql(\nu_t^{*,\xi})-P_tBh(\mu_t^{*,\xi})]dt+\Lambda_t^{0,*,\xi}dW^0_t,\\
\varphi_T^{*,\xi}=&~Gg(\nu_T^{*,\xi}),
\end{aligned}
\right.
\end{equation}
and $\alpha^{*,\xi}$ is given by
\begin{equation}\label{alpha*}
\alpha^{*,\xi}_t:=-R^{-1}B(P_tx^{*,\xi}_t+\varphi_t^{*,\xi})-h(\mu_t^{*,\xi}).
\end{equation}
Taking the conditional expectation of \eqref{alpha*} on $\mathcal{F}_t^{W^0}$, we have 
\begin{equation*}
\mu_t^{*,\xi}=-R^{-1}BP_t\nu_t^{*,\xi}-R^{-1}B\varphi_t^{*,\xi}-h(\mu_t^{*,\xi}),
\end{equation*}
which leads to
\[\mu_t^{*,\xi}+h(\mu^{*,\xi}_t)=-R^{-1}BP_t\nu_t^{*,\xi}-R^{-1}B\varphi_t^{*,\xi}.\]
Applying Assumption (B), we can use the inverse function theorem to derive that there exists a uniformly Lipschitz continuous function $\rho:\mathbb R\to\mathbb R$ such that 
\begin{equation}\label{rho}\mu_t^{*,\xi}=\rho(-R^{-1}BP_t\nu_t^{*,\xi}-R^{-1}B\varphi_t^{*,\xi})  \end{equation}
where
\begin{equation}\label{eq:rho'}
|\rho^\prime|= \left|\frac{1}{1+h^\prime}\right|\leq \frac{1}{\varepsilon_0}. 
\end{equation}
For simplicity, we define a function $k:[0,T]\times\mathbb R\times\mathbb R\to\mathbb R$ by
\begin{equation}\label{k}
k(t,x,y):=\rho(-R^{-1}BP_tx-R^{-1}By)\quad\text{for any $(t,x,y)\in[0,T]\times\mathbb R\times\mathbb R$}
\end{equation}
and thus \eqref{rho} becomes
\[
\mu_t^{*,\xi}=k(t,\nu_t^{*,\xi},\varphi_t^{*,\xi}).
\]
We note that \eqref{eq:rho'} implies $k(t,\cdot,\cdot)$ is uniformly Lipschitz continuous in $x$ and $y$, uniformly in $[0,T]$. Taking the conditional expectation of \eqref{fstate*} on $\mathcal{F}^{W^0}_t$, we have 
\begin{equation*}
\begin{aligned}
\nu_t^{*,\xi}=&\mathbb{E}[x_t^{*,\xi}|\mathcal{F}^{W^0}_t]\\
=&\mathbb{E}[\xi]+\int_0^t[(A-B^2R^{-1}P_s)\mathbb{E}[x_s^{*,\xi}|\mathcal{F}_s^{W^0}]-B^2R^{-1}\varphi_s^{*,\xi}-Bh(\mu_s^{*,\xi})+f(\nu_s^{*,\xi})+b(\mu_s^{*,\xi})]ds+\int_0^t\sigma_0dW^0_s,
\end{aligned}
\end{equation*}
which implies 
\begin{equation}\label{nu*}
\left\{
\begin{aligned}
d\nu_t^{*,\xi}=&[(A-B^2R^{-1}P_t)\nu_t^{*,\xi}-B^2R^{-1}\varphi_t^{*,\xi}-Bh(\mu_t^{*,\xi})+f(\nu_t^{*,\xi})+b(\mu_t^{*,\xi})]dt+\sigma_0dW^0_t,\\
\nu_0^{*,\xi}=&~\mathbb{E}[\xi].
\end{aligned}
\right.
\end{equation}
Substituting $\mu_t^{*,\xi}=k(t,\nu_t^{*,\xi},\varphi_t^{*,\xi})$ into \eqref{varphi*} and \eqref{nu*}, we derive the following Nash certainty equivalence (NCE) system 
\begin{equation}\label{NCE}
 \left\{
 \begin{aligned}
 d\nu_t^{*,\xi}=&~\Big[(A-B^2R^{-1}P_t)\nu_t^{*,\xi}-B^2R^{-1}\varphi_t^{*,\xi}-Bh\Big(k(t,\nu_t^{*,\xi},\varphi_t^{*,\xi})\Big)\\
 &+f(\nu_t^{*,\xi})+b\Big(k(t,\nu_t^{*,\xi},\varphi_t^{*,\xi})\Big)\Big]dt+\sigma_0dW^0_t,\\
 d\varphi_t^{*,\xi}=&-\Big[(A-B^2R^{-1}P_t)\varphi_t^{*,\xi}+P_tf(\nu_t^{*,\xi})+P_tb\Big(k(t,\nu_t^{*,\xi},\varphi_t^{*,\xi})\Big)\\&+Ql(\nu_t^{*,\xi})-P_tBh\Big(k(t,\nu_t^{*,\xi},\varphi_t^{*,\xi})\Big)\Big]dt+\Lambda_t^{0,*,\xi}dW^0_t,\\
\nu_0^{*,\xi}=&~\mathbb{E}[\xi],~\varphi_T^{*,\xi}=Gg(\nu_T^{*,\xi}).
  \end{aligned}
 \right.
 \end{equation}
The NCE system \eqref{NCE} is a coupled FBSDE driven by the common noise $W^0$ only. In fact, the NCE system \eqref{NCE} is non-degenerate since $\sigma_0>0$. Different from most of the existing literature, we shall prove the well-posedness of NCE system \eqref{NCE} without any \emph{monotonicity conditions} by using the ideas in \cite{Delarue2002,MaProtterYong1994}.

\begin{theorem}\label{NCEwellposed}
Suppose that Assumptions (A), (B) hold. For any $\xi\in L^2_{\mathcal{F}_0}$, there exists a unique strong solution $(\nu^{*,\xi},\varphi^{*,\xi}, \Lambda^{*,\xi})$ to the NCE system \eqref{NCE}. 
\end{theorem}
\begin{proof}
We would like to first derive an equivalent system of the NCE system \eqref{NCE}. Let $P$ be the unique solution to the Riccati equation \eqref{P}. Define
\begin{equation}\label{eta} \eta_t:=\exp\int_t^T(A-B^2R^{-1}P_s)ds.\end{equation}
Then there exists some $\varepsilon_1>0$ such that
 \begin{equation}\label{eq:eta} \varepsilon_1\leq\eta_t\leq \frac{1}{\varepsilon_1}.
 \end{equation} 
 Suppose that $(\nu^{*,\xi},\varphi^{*,\xi},\Lambda^{0,*,\xi})$ is a strong solution to \eqref{NCE}. We introduce the following transform:
 \begin{equation}\label{transform}\widetilde{\nu}_t^{*,\xi}=\eta_t \nu_t^{*,\xi},\quad \widetilde{\varphi}_t^{*,\xi}=\eta^{-1}_t\varphi_t^{*,\xi} ,\quad\widetilde{\Lambda}_t^{0,*,\xi}=\eta^{-1}_t\Lambda_t^{0,*,\xi}, \text{ for any } t\in[0,T].\end{equation}
 To simply the notations, we first define a function $\widetilde k:[0,T]\times\mathbb R\times\mathbb R\to\mathbb R$ by
  \begin{equation}\label{tildek}
 \widetilde k(t,x,y):=\rho(-R^{-1}BP_t\eta_t^{-1}x-R^{-1}B\eta_t y)\quad\text{for any $(t,x,y)\in[0,T]\times\mathbb R\times\mathbb R$}
 \end{equation}
 and then we define functions $\widetilde h, \widetilde b:[0,T]\times\mathbb R\times\mathbb R\to \mathbb R$ by
 \begin{equation}\label{tildehtildeb}
  \widetilde h(t,x,y)=h( \widetilde k(t,x,y))\quad\text{and}\quad   \widetilde b(t,x,y)=b( \widetilde k(t,x,y)).
 \end{equation}
By Assumptions (A), (B) and using \eqref{eq:eta}, we know that $ \widetilde h(t,\cdot,\cdot), \widetilde b(t,\cdot,\cdot)$ are uniformly Lipschitz continuous in $x$ and $y$, uniformly in $t\in[0,T]$. It can be verified that 
 $(\widetilde{\nu}^{*,\xi},\widetilde{\varphi}^{*,\xi},\widetilde{\Lambda}^{0,*,\xi})$ is a strong solution to the following FBSDE: 
 \begin{equation}\label{EFBSDE}
 \left\{
 \begin{aligned}
 d\widetilde{\nu}_t^{*,\xi}=&~\Big[-\eta_t^2B^2R^{-1}\widetilde{\varphi}_t^{*,\xi}-\eta_tB \widetilde h( t, \widetilde\nu_t^{*,\xi}, \widetilde\varphi_t^{*,\xi})\\
 &+\eta_tf(\eta_t^{-1}\widetilde{\nu}_t^{*,\xi})+\eta_t\widetilde b( t, \widetilde\nu_t^{*,\xi}, \widetilde\varphi_t^{*,\xi})\Big]dt+\eta_t\sigma_0dW^0_t,\\
 d\widetilde{\varphi}_t^{*,\xi}=&-\Big[\eta_t^{-1}P_tf(\eta_t^{-1}\widetilde{\nu}_t^{*,\xi})+\eta_t^{-1}P_t\widetilde b( t, \widetilde\nu_t^{*,\xi}, \widetilde\varphi_t^{*,\xi})\\
 &+\eta_t^{-1}Ql(\eta_t^{-1}\widetilde{\nu}_t^{*,\xi})-\eta_t^{-1}P_tB\widetilde h( t, \widetilde\nu_t^{*,\xi}, \widetilde\varphi_t^{*,\xi})\Big]dt+\widetilde{\Lambda}_t^{0,*,\xi}dW^0_t,\\
\widetilde{\nu}_0^{*,\xi}=&~\eta_0\mathbb{E}[\xi],~\widetilde{\varphi}_T^{*,\xi}=Gg(\widetilde{\nu}_T^{*,\xi}).
  \end{aligned}
 \right.
 \end{equation}
 Similarly, we can check that if $(\widetilde{\nu}^{*,\xi},\widetilde{\varphi}^{*,\xi},\widetilde{\Lambda}^{0,*,\xi})$ is a strong solution to \eqref{EFBSDE}, then $(\nu^{*,\xi},\varphi^{*,\xi},\Lambda^{0,*,\xi})$ is a strong solution to \eqref{NCE}. Therefore, we show the equivalence between the NCE system \eqref{NCE} and the FBSDE \eqref{EFBSDE}.

We note that the coefficients in \eqref{EFBSDE} satisfy all the assumptions in \cite[Theorem 2.6]{Delarue2002}. Therefore, we derive the well-posedness of the FBSDE \eqref{EFBSDE}. By the equivalence, we show the well-posedness of the NCE system \eqref{NCE}.
\end{proof}
\begin{remark}
(i) It is worth noting that, although $\sigma_0>0$ in the NCE system \eqref{NCE}, its coefficients do not satisfy the corresponding growth conditions in \cite[Theorem 2.6]{Delarue2002}. Therefore, we can not directly apply the well-posedness result \cite[Theorem 2.6]{Delarue2002} to the NCE system \eqref{NCE}.\\
(ii) The non-degeneracy assumption $\sigma_0>0$ is necessary for the above global well-posedness result for the NCE system \eqref{NCE} without any monotonicity conditions. Otherwise, it is possible to construct an example such that the NCE system is not well-posed, see, e.g. \cite[Section 4.2]{Tchuendom2018}. 
\end{remark}

\subsection{Master equation}
Throughout this subsection, we assume that Assumptions (A), (B), (C) hold. We remind that the functions $\rho$ and $k$ are given in \eqref{rho} and \eqref{k}.

In this subsection, we consider the well-posedness of the following master equation corresponding to the Problem (MF):
   \begin{equation}\label{master}
 \left\{
 \begin{aligned}
&\partial_tV(t,x,\nu)-\frac{1}{2}B^2R^{-1}|\partial_xV(t,x,\nu)|^2+\partial_xV(t,x,\nu)\Big[Ax-Bh\Big(\rho\big(-R^{-1}B\mathbb E\big[\partial_xV(t,\xi,\nu)\big]\big)\Big)\\
 &+f(\nu)+b\Big(\rho\big(-R^{-1}B\mathbb E\big[\partial_xV(t,\xi,\nu)\big]\big)\Big)\Big]
 +\frac{1}{2}\partial_{xx}V(t,x,\nu)(\sigma^2+\sigma_0^2)+\partial_{\nu}V(t,x,\nu)\Big[A\nu-B^2R^{-1}\mathbb E\big[\partial_xV(t,\xi,\nu)\big]\\
 &-Bh\Big(\rho\big(-R^{-1}B\mathbb E\big[\partial_xV(t,\xi,\nu)\big]\big)\Big)+f(\nu)+b\Big(\rho\big(-R^{-1}B\mathbb E\big[\partial_xV(t,\xi,\nu)\big]\big)\Big)\Big]\\
 &+\frac{1}{2}\partial_{\nu\nu}V(t,x,\nu)\sigma_0^2+\partial_{x\nu}V(t,x,\nu)\sigma_0^2+\frac{1}{2}Q\big(x+l(\nu)\big)^2=0,\\
 &V(T,x,\nu)=\frac{1}{2}G\big(x+g(\nu)\big)^2,
 \end{aligned}
 \right.
 \end{equation}
where $\xi\in L_{\mathcal{F}_0}^2$ with $\mathbb E[\xi]=\nu$.

We shall study the well-posedness of the above master equation \eqref{master} via the following vectorial master equation 
  \begin{equation}\label{eq:vecmaster}
  \left\{
 \begin{aligned}
 &\partial_tU(t,x,\nu)+\partial_xU(t,x,\nu)\Big[Ax-B^2R^{-1}U(t,x,\nu)-Bh\Big(\rho\big(-R^{-1}B\mathbb E\big[U(t,\xi,\nu)\big]\big)\Big)\\
 &+f(\nu)+b\Big(\rho\big(-R^{-1}B\mathbb E\big[U(t,\xi,\nu)\big]\big)\Big)\Big]
 +\frac{1}{2}\partial_{xx}U(t,x,\nu)(\sigma^2+\sigma_0^2)+\partial_{\nu}U(t,x,\nu)\Big[A\nu-B^2R^{-1}\mathbb E\big[U(t,\xi,\nu)\big]\\
 &-Bh\Big(\rho\big(-R^{-1}B\mathbb E\big[U(t,\xi,\nu)\big]\big)\Big)+f(\nu)+b\Big(\rho\big(-R^{-1}B\mathbb E\big[U(t,\xi,\nu)\big]\big)\Big)\Big]\\
 &+\frac{1}{2}\partial_{\nu\nu}U(t,x,\nu)\sigma_0^2+\partial_{x\nu}U(t,x,\nu)\sigma_0^2+AU(t,x,\nu)+Q(x+l(\nu))=0,\\
 &U(T,x,\nu)=G(x+g(\nu)).
 \end{aligned}
 \right.
 \end{equation}
The solution $U$ to the vectorial master equation serves as the decoupling field of the stochastic Hamiltonian system \eqref{H}. Thanks to the LQ setting, the solution $U$
 has the form $U(t,x,\nu)=P_tx +\Phi(t,\nu)$ where $P$ is the unique solution to the Riccati equation \eqref{P} and $\Phi$ solves
\begin{equation}\label{Phi}
\left\{
\begin{aligned}
&\partial_t\Phi(t,\nu)+\partial_\nu\Phi(t,\nu)\Big[(A-B^2R^{-1}P_t)\nu-B^2R^{-1}\Phi(t,\nu)-Bh\Big(k(t,\nu,\Phi(t,\nu))\Big)\\
&+f(\nu)+b\Big(k(t,\nu,\Phi(t,\nu))\Big)\Big]+\frac{1}{2}\partial_{\nu\nu}\Phi(t,\nu)\sigma_0^2+(A-B^2R^{-1}P_t)\Phi(t,\nu)+P_tf(\nu)\\
&+P_tb\Big(k(t,\nu,\Phi(t,\nu))\Big)+Ql(\nu)-P_tBh\Big(k(t,\nu,\Phi(t,\nu))\Big)=0,\\
&\Phi(T,\nu)=Gg(\nu).
\end{aligned}
\right.
\end{equation}
The solution $\Phi$ above is the decoupling field of the NCE system \eqref{NCE}. In light of the proof of Theorem \ref{NCEwellposed}, we shall consider the following FBSDE: for any $(t_0,\nu_0)\in[0,T]\times\mathbb R$
\begin{equation}\label{EFBSDEt0}
 \left\{
 \begin{aligned}
 d\widetilde{\nu}_t^{t_0,\nu_0}=&~\Big[-\eta_t^2B^2R^{-1}\widetilde{\varphi}_t^{t_0,\nu_0}-\eta_tB\widetilde h(t,\widetilde{\nu}_t^{t_0,\nu_0},\widetilde{\varphi}_t^{t_0,\nu_0})\\
 &+\eta_tf(\eta_t^{-1}\widetilde{\nu}_t^{t_0,\nu_0})+\eta_t\widetilde b(t,\widetilde{\nu}_t^{t_0,\nu_0},\widetilde{\varphi}_t^{t_0,\nu_0})\Big]dt+\eta_t\sigma_0dW^0_t,\\
d\widetilde{\varphi}_t^{t_0,\nu_0}=&-\Big[\eta_t^{-1}P_tf(\eta_t^{-1}\widetilde{\nu}_t^{t_0,\nu_0})+\eta_t^{-1}P_t\widetilde b(t,\widetilde{\nu}_t^{t_0,\nu_0},\widetilde{\varphi}_t^{t_0,\nu_0})\\
 &+\eta_t^{-1}Ql(\eta_t^{-1}\widetilde{\nu}_t^{t_0,\nu_0})-\eta_t^{-1}P_tB\widetilde h (t,\widetilde{\nu}_t^{t_0,\nu_0},\widetilde{\varphi}_t^{t_0,\nu_0})\Big]dt+\widetilde{\Lambda}_t^{0,t_0,\nu_0}dW^0_t,\\
\widetilde{\nu}_{t_0}^{t_0,\nu_0}=&~\nu_0,~\widetilde{\varphi}_T^{t_0,\nu_0}=\widetilde{g}(\widetilde{\nu}_T^{t_0,\nu_0}),
 \end{aligned}
 \right.
 \end{equation}
where $\widetilde{g}\in C^2(\mathbb R)$, $\eta_t$ is given in \eqref{eta} and $\widetilde{h},\widetilde{b}$ are given in \eqref{tildehtildeb}. The decoupling field $\widetilde{\Phi}$ of the FBSDE \eqref{EFBSDEt0} corresponds to the following PDE:
\begin{equation}\label{tildePhi}
\left\{
\begin{aligned}
&\partial_t\widetilde{\Phi}(t,\nu)+\partial_\nu\widetilde{\Phi}(t,\nu)\Big[-\eta_t^2B^2R^{-1}\widetilde\Phi(t,\nu)-\eta_tB\widetilde{h}\big(t,\nu,\widetilde\Phi(t,\nu)\big)+\eta_t f(\eta_t^{-1}\nu)+\eta_t\widetilde{b}(t,\nu,\widetilde\Phi(t,\nu))\Big]\\
&+\frac{1}{2}\partial_{\nu\nu}\widetilde\Phi(t,\nu)\eta_t^2\sigma_0^2+\eta_t^{-1}P_tf(\eta_t^{-1}\nu)+\eta_t^{-1}P_t\widetilde{b}\big(t,\nu,\widetilde\Phi(t,\nu)\big)+\eta_t^{-1}Ql(\eta_t^{-1}\nu)-\eta_t^{-1}P_tB\widetilde{h}(t,\nu,\widetilde\Phi(t,\nu))=0,\\
&\widetilde{\Phi}(T,\nu)=\widetilde{g}(\nu).
\end{aligned}
\right.
\end{equation}
We need to show that $\widetilde{\Phi}\in C^{1,2}([0,T]\times\mathbb R)$ to verify $\widetilde{\Phi}$ is indeed a classical solution to \eqref{tildePhi}. Let us consider the following FBSDE on $[t_0,T]$, which can be interpreted as a formal differentiation of \eqref{EFBSDEt0} with respect to $\nu_0$:
\begin{equation}\label{nablaEFBSDEt0}
 \left\{
 \begin{aligned}
 d\nabla\widetilde{\nu}_t^{t_0,\nu_0}=&~\Big[-\eta_t^2B^2R^{-1}\nabla \widetilde{\varphi}_t^{t_0,\nu_0}-\eta_tB\partial_{\nu}\widetilde h(t,\widetilde{\nu}_t^{t_0,\nu_0},\widetilde{\varphi}_t^{t_0,\nu_0})\nabla \widetilde{\nu}_t^{t_0,\nu_0}-\eta_tB\partial_{\varphi}\widetilde h(t,\widetilde{\nu}_t^{t_0,\nu_0},\widetilde{\varphi}_t^{t_0,\nu_0})\nabla \widetilde{\varphi}_t^{t_0,\nu_0}\\
 &+f'(\eta_t^{-1}\widetilde{\nu}_t^{t_0,\nu_0})\nabla \widetilde{\nu}_t^{t_0,\nu_0}+\eta_t\partial_\nu\widetilde b(t,\widetilde{\nu}_t^{t_0,\nu_0},\widetilde{\varphi}_t^{t_0,\nu_0})\nabla \widetilde{\nu}_t^{t_0,\nu_0}+\eta_t\partial_{\varphi}\widetilde b(t,\widetilde{\nu}_t^{t_0,\nu_0},\widetilde{\varphi}_t^{t_0,\nu_0})\nabla \widetilde{\varphi}_t^{t_0,\nu_0}\Big]dt,\\
d\nabla\widetilde{\varphi}_t^{t_0,\nu_0}=&-\Big[\eta_t^{-2}P_tf'(\eta_t^{-1}\widetilde{\nu}_t^{t_0,\nu_0})\nabla\widetilde{\nu}_t^{t_0,\nu_0}+\eta_t^{-1}P_t\partial_{\nu}\widetilde b(t,\widetilde{\nu}_t^{t_0,\nu_0},\widetilde{\varphi}_t^{t_0,\nu_0})\nabla \widetilde{\nu}_t^{t_0,\nu_0}+\eta_t^{-1}P_t\partial_{\varphi}\widetilde b(t,\widetilde{\nu}_t^{t_0,\nu_0},\widetilde{\varphi}_t^{t_0,\nu_0})\nabla \widetilde{\varphi}_t^{t_0,\nu_0}\\
 &+\eta_t^{-2}Ql'(\eta_t^{-1}\widetilde{\nu}_t^{t_0,\nu_0})\nabla \widetilde{\nu}_t^{t_0,\nu_0}-\eta_t^{-1}P_tB\partial_{\nu}\widetilde h (t,\widetilde{\nu}_t^{t_0,\nu_0},\widetilde{\varphi}_t^{t_0,\nu_0})\nabla \widetilde{\nu}_t^{t_0,\nu_0}\\
 &-\eta_t^{-1}P_tB\partial_{\varphi}\widetilde h (t,\widetilde{\nu}_t^{t_0,\nu_0},\widetilde{\varphi}_t^{t_0,\nu_0})\nabla \widetilde{\varphi}_t^{t_0,\nu_0}\Big]dt+\nabla \widetilde{\Lambda}_t^{0,t_0,\nu_0}dW^0_t,\\
\nabla\widetilde{\nu}_{t_0}^{t_0,\nu_0}=&~1,~\nabla\widetilde{\varphi}_T^{t_0,\nu_0}=\widetilde{g}'(\widetilde{\nu}_T^{t_0,\nu_0})\nabla \widetilde{\nu}_T^{t_0,\nu_0}.
 \end{aligned}
 \right.
 \end{equation}
\begin{theorem}\label{thm:tildePhi}
Let Assumptions (A), (B), (C) hold and let $\widetilde{g}\in C^2(\mathbb R)$ be a bounded function with bounded 1st and 2nd order derivatives. Then\\
(i) The FBSDE \eqref{EFBSDEt0} is well-posed on $[t_0,T]$ for any $(t_0,\nu_0)\in[0,T]\times\mathbb R$;\\
(ii) Define $\widetilde \Phi(t_0,\nu_0)=\widetilde{\varphi}_{t_0}^{t_0,\nu_0}$ for any $(t_0,\nu_0)\in[0,T]\times\mathbb R$. Then $\widetilde \Phi\in C^{1,2}([0,T]\times\mathbb R)$, with bounded $\partial_{\nu}\widetilde \Phi, \partial_{\nu\nu}\widetilde \Phi$, is the unique classical solution to \eqref{tildePhi}.
\end{theorem}
\begin{proof}
(i) Note that the coefficients in \eqref{EFBSDEt0} meet the assumptions in \cite[Theorem 2.6]{Delarue2002}. Therefore, the FBSDE \eqref{EFBSDEt0} is well-posed for any $(t_0,\nu_0)\in [0,T]\times\mathbb R$.\\
(ii) \textbf{Existence:} Define $\widetilde \Phi(t_0,\nu_0)=\widetilde{\varphi}_{t_0}^{t_0,\nu_0}$ for any $(t_0,\nu_0)\in[0,T]\times\mathbb R$. Applying \cite[Corollary 2.8]{Delarue2002}, there exists some $\widetilde{C}_0>0$ such that $\|\partial_{\nu}\widetilde\Phi\|_{L^\infty([0,T]\times\mathbb R)}\leq \widetilde C_0$.

Now we would like to show that $\partial_{\nu}\widetilde\Phi,\partial_{\nu\nu}\widetilde\Phi\in C^0([0,T]\times\mathbb R)$ are bounded. Let us consider the linear FBSDE \eqref{nablaEFBSDEt0} on $[t_0,T]$ for any $t_0\in[0,T)$. Note that Assumptions (A), (B), (C) imply that all the coefficients in FBSDE \eqref{nablaEFBSDEt0} are bounded by some chosen $\widetilde C_1\geq \widetilde C_0$. By standard FBSDE arguments (see e.g. \cite{Zhang2017}), there exists some $\delta_0>0$ depending on $\widetilde C_1>0$ such that the FBSDE \eqref{nablaEFBSDEt0} is well-posed on $[T-\delta_0,T]$. We then consider the FBSDE \eqref{nablaEFBSDEt0} with $\widetilde{g}(\cdot)$ replaced by $\widetilde\Phi(T-\delta_0,\cdot)$. Since $\|\partial_{\nu}\widetilde\Phi(\cdot,\cdot)\|_{L^\infty([0,T]\times\mathbb R)}\leq \widetilde C_1$, the FBSDE \eqref{nablaEFBSDEt0} is also well-posed on $[T-2\delta_0,T-\delta_0]$. After finite steps, we are able to show that the FBSDE \eqref{nablaEFBSDEt0} is well-posed on $[t_0,T]$. It implies that $\partial_{\nu}\widetilde{\Phi}\in C^0([0,T]\times\mathbb R)$. By differentiating \eqref{nablaEFBSDEt0} in $\nu_0$ and using $\|\partial_{\nu}\widetilde\Phi\|_{L^\infty([0,T]\times\mathbb R)}\leq \widetilde C_0$, we can further show that $\partial_{\nu\nu}\widetilde \Phi\in C^0([0,T]\times\mathbb R)$ and there exists $\widetilde C_2>0$ such that $\|\partial_{\nu\nu}\widetilde \Phi\|_{L^\infty([0,T]\times\mathbb R)}\leq \widetilde{C}_2$.

With the uniform boundedness of $\partial_{\nu}\widetilde \Phi$ and the continuity of $\partial_{\nu}\widetilde \Phi,\partial_{\nu\nu}\widetilde \Phi$, it is rather standard to verify that $\widetilde\Phi\in C^{1,2}([0,T]\times\mathbb R)$ and it satisfies the PDE \eqref{tildePhi}.

\textbf{Uniqueness:} Suppose that $\bar\Phi\in C^{1,2}([0,T]\times\mathbb R)$ is another classical solution to \eqref{tildePhi} with bounded $\partial_\nu\bar\Phi$. For any $(t_0,\nu_0)\in[0,T]\times\mathbb R$, we first consider the following well-posed stochastic differential equation (SDE)
\begin{equation*}
 \left\{
 \begin{aligned}
d\bar{\nu}_t^{t_0,\nu_0}=&~\Big[-\eta_t^2B^2R^{-1}\bar\Phi(t,\bar{\nu}_t^{t_0,\nu_0})-\eta_tB\widetilde h(t,\bar{\nu}_t^{t_0,\nu_0},\bar\Phi(t,\bar{\nu}_t^{t_0,\nu_0}))\\
 &+\eta_tf(\eta_t^{-1}\bar{\nu}_t^{t_0,\nu_0})+\eta_t\widetilde b(t,\bar{\nu}_t^{t_0,\nu_0},\bar\Phi(t,\bar{\nu}_t^{t_0,\nu_0}))\Big]dt+\eta_t\sigma_0dW^0_t,\\
\bar\nu_{t_0}^{t_0,\nu_0}=&~\eta_0\nu_0.
  \end{aligned}
 \right.
 \end{equation*}
 Let $\bar\varphi_t^{t_0,\nu_0}:=\bar\Phi(t,\bar\nu_t^{t_0,\nu_0})$. Since $\bar\Phi$ is a classical solution to \eqref{tildePhi}, it can be easily checked that $(\bar\varphi^{t_0,\nu_0},\bar\Lambda^{0,t_0,\nu_0})$ solves the BSDE in \eqref{EFBSDEt0} where $\bar\Lambda_t^{0,t_0,\nu_0}=\eta_t\sigma_0\partial_\nu \bar\Phi(t,\bar{\nu}_t^{t_0,\nu_0})$. Therefore we have verified that $(\bar\nu^{t_0,\nu_0},\bar\varphi^{t_0,\nu_0},\bar\Lambda^{t_0,\nu_0})$ is a strong solution to FBSDE \eqref{EFBSDEt0}. Therefore, the uniqueness result follows by the well-posedness of the FBSDE \eqref{EFBSDEt0}.
\end{proof}
\begin{theorem}\label{thm:Phi}
Let Assumptions (A), (B), (C) hold. Then the PDE \eqref{Phi} admits a unique classical solution $\Phi$ with bounded $\partial_\nu\Phi,\partial_{\nu\nu}\Phi$.
\end{theorem}
\begin{proof} It can be easily checked that the equations \eqref{Phi} and \eqref{tildePhi} with $\widetilde{g}(\nu)=\eta_TGg(\eta_T\nu)$ are equivalent, i.e. $\Phi(t,\nu):=\eta_{t}^{-1}\widetilde\Phi(t,\eta_t^{-1}\nu)$ is a solution to \eqref{Phi} if and only if $\widetilde \Phi(t,\nu)$ is a solution to \eqref{tildePhi} with $\widetilde{g}(\nu)=\eta_TGg(\eta_T\nu)$. Then, by Theorem \ref{thm:tildePhi}, we obtain 
the PDE \eqref{Phi} admits a unique classical solution $\Phi$. Together with \eqref{eq:eta}, the boundedness of $\partial_\nu\widetilde\Phi,\partial_{\nu\nu}\widetilde\Phi$ implies the boundedness of $\partial_\nu\Phi,\partial_{\nu\nu}\Phi$.
\end{proof}

\begin{theorem}\label{thm:vecmaster}
Let Assumptions (A), (B), (C) hold and let $\Phi$ be given in Theorem \ref{thm:Phi}. Then 
\begin{equation}\label{eq:solvecmaster}
U(t,x,\nu):=P_tx+\Phi(t,\nu)
\end{equation} is the unique classical solution to the vectorial master equation \eqref{eq:vecmaster} with bounded $\partial_x U,\partial_{\nu}U,\partial_{\nu\nu}U$.
\end{theorem}
\begin{proof}
 \textbf{Existence:} We first check that $U$ satisfies the terminal condition
\begin{equation*}
U(T,x,\nu)=P_Tx+\Phi(T,\nu)=G(x+g(\nu)).
\end{equation*}
At the end we verify that $U$ indeed satisfies the vectorial master equation \eqref{eq:vecmaster}. Using \eqref{P} and \eqref{Phi}, we obtain
\begin{eqnarray*}
&&\partial_t U(t,x,\nu)=\dot{P}_tx+\partial_t\Phi(t,\nu)\\
&=&(-2AP_t+B^2R^{-1}P_t^2-Q)x-\partial_{\nu}\Phi(t,\nu)\big[(A-B^2R^{-1}P_t)\nu
-B^2R^{-1}\Phi(t,\nu)\\
&&-Bh(\rho(-R^{-1}BP_t\nu-R^{-1}B\Phi(t,\nu))+f(\nu)+b(\rho(-R^{-1}BP_t\nu-R^{-1}B\Phi(t,\nu)))\big]\\
&&-\frac{1}{2}\partial_{\nu\nu}\Phi(t,\nu)\sigma_0^2-(A-B^2R^{-1}P_t)\Phi(t,\nu)-P_tf(\nu)-P_tb\Big(\rho\big(-R^{-1}BP_t\nu-R^{-1}B\Phi(t,\nu)\big)\Big)\\
&&-Ql(\nu)-P_tBh\Big(\rho\big(-R^{-1}BP_t\nu-R^{-1}B\Phi(t,\nu)\big)\Big).
\end{eqnarray*}
Moreover, we have
\[
\partial_xU(t,x,\nu)=P_t,\quad\partial_{xx}U(t,x,\nu)=\partial_{x\nu}U(t,x,\nu)=0,\quad\partial_{\nu}U(t,x,\nu)=\partial_{\nu}\Phi(t,\nu),\quad\partial_{\nu\nu}U(t,x,\nu)=\partial_{\nu\nu}\Phi(t,\nu)
\]
are bounded. Plugging the above terms into \eqref{eq:vecmaster} and using \eqref{eq:solvecmaster}, it is straightforward to show that $U$ is a classical solution to the vectorial master equation \eqref{eq:vecmaster}.

\textbf{Uniqueness:} We remind that the solution $U$ to \eqref{eq:vecmaster} serves as the decoupling field of the stochastic Hamiltonian system \eqref{H}. Following the uniqueness argument in Theorem \ref{thm:tildePhi}, the well-posedness of the stochastic Hamiltonian system \eqref{H} implies the uniqueness of a solution to the vectorial master equation \eqref{eq:vecmaster}.
\end{proof}

\begin{theorem}
Let Assumptions (A), (B), (C) hold. Then the master equation \eqref{master} admits a unique classical solution $V$ with bounded $\partial_{xx}V,\partial_{x\nu}V,\partial_{x\nu\nu}V$.
\end{theorem}
\begin{proof}
 \textbf{Existence:}  Let $(t_0,x_0,\nu_0)\in[0,T]\times\mathbb R\times \mathbb R$ and let $U$ be the unique classical solution to the vectorial master equation \eqref{eq:vecmaster} given in Theorem \ref{thm:vecmaster}. For any $\xi\in L^2_{\mathcal{F}_{t_0}}$ with $\mathbb E[\xi]=\nu_0$, the SDE
\begin{equation}\label{eq:nv*}
\left\{
\begin{aligned}
dx_t^{*,t_0,\xi}=&~[Ax_t^{*,t_0,\xi}-B^2R^{-1}U(t,x_t^{*,t_0,\xi},\nu_t^{*,t_0,\xi})-Bh(\mu_t^{*,t_0,\xi})+f(\nu_t^{*,t_0,\xi})+b(\mu_t^{*,t_0,\xi})]dt\\
&+\sigma dW_t+\sigma_0dW^0_t,\\
x_{t_0}^{*,t_0,\xi}=&~\xi,
\end{aligned}
\right.
\end{equation}
is well-posed on $[t_0,T]$ where
\begin{equation}\label{munut0}
\mu_t^{*,t_0,\xi}:=\rho\big(-R^{-1}B\mathbb E\big[U(t,x_t^{*,t_0,\xi},\nu_t^{*,t_0,\xi})|\mathcal{F}_t^{W^0}\big]\big)\quad\text{and}\quad \nu_t^{*,t_0,\xi}:=\mathbb E\big[x_t^{*,t_0,\xi}|\mathcal{F}_t^{W^0}\big].
\end{equation}
Given $(\mu^{*,t_0,\xi},\nu^{*,t_0,\xi})$ above, the SDE: for any $x_0\in\mathbb R$
\begin{equation}\label{eq:nvx*}
\left\{
\begin{aligned}
dx_t^{*,t_0,x_0,\xi}=&~[Ax_t^{*,t_0,x_0,\xi}-B^2R^{-1}U(t,x_t^{*,t_0,x_0,\xi},\nu_t^{*,t_0,\xi})-Bh(\mu_t^{*,t_0,\xi})+f(\nu_t^{*,t_0,\xi})+b(\mu_t^{*,t_0,\xi})]dt\\
&+\sigma dW_t+\sigma_0dW^0_t,\\
x_{t_0}^{*,t_0,x_0,\xi}=&~x_0,
\end{aligned}
\right.
\end{equation}
is also well-posed on $[t_0,T]$. Then we can define
\begin{equation}\label{V}
\begin{aligned}
&V(t_0,x_0,\nu_0)\\
:=&\frac{1}{2}\mathbb{E}\Big\{\int_{t_0}^T\Big[Q\left(x_t^{*,t_0,x_0,\xi}+l(\nu_t^{*,t_0,\xi})\right)^2+R\left(\alpha_t^{*,t_0,x_0,\xi}+h(\mu_t^{*,t_0,\xi})\right)^2\Big]dt+G\left(x_T^{*,t_0,x_0,\xi}+g(\mu_T^{*,t_0,\xi})\right)^2\Big\}
\end{aligned}
\end{equation}
where
\[
\alpha^{*,t_0,x_0,\xi}_t:=-BR^{-1}U(t,x_t^{*,t_0,x_0,\xi},\nu_t^{*,t_0,\xi})-h(\mu_t^{*,t_0,\xi}).
\]
We note that $U\in C^{1,2,2}([0,T]\times\mathbb R\times\mathbb R)$ implies $V\in C^{1,2,2}([0,T]\times\mathbb R\times\mathbb R)$. By the derivation of $U$, we know that
\begin{equation*}
\begin{aligned}
&V(t_0,x_0,\nu_0)\\
=&\underset{\alpha\in\mathcal{U}_{ad}(t_0,T)}{\inf}\frac{1}{2}\mathbb{E}\Big\{\int_{t_0}^T\Big[Q\left(x_t^{t_0,x_0,\xi,\alpha}+l(\nu_t^{*,t_0,\xi})\right)^2+R\left(\alpha_t+h(\mu_t^{*,t_0,\xi})\right)^2\Big]dt+G\left(x_T^{t_0,x_0,\xi,\alpha}+g(\mu_T^{*,t_0,\xi})\right)^2\Big\}
\end{aligned}
\end{equation*}
where
\[
\left\{
\begin{aligned}
dx_t^{t_0,x_0,\xi,\alpha}=&~[Ax_t^{t_0,x_0,\xi,\alpha}+B\alpha_t+f(\nu_t^{*,t_0,\xi})+b(\mu_t^{*,t_0,\xi})]dt+\sigma dW_t+\sigma_0dW^0_t,\\
x_{t_0}^{t_0,x_0,\xi,\alpha}=&~x_0.
\end{aligned}
\right.
\]

Following the standard dynamic programming arguments, it is straightforward to check that $V$ satisfies the master equation \eqref{master} and $\partial_xV(t,x,\nu)=U(t,x,\nu)$.

 \textbf{Uniqueness:} Suppose that $\bar V\in C^{1,2,2}([0,T]\times\mathbb R\times\mathbb R)$ is another classical solution to the master equation \eqref{master} with bounded $\partial_{xx}\bar V,\partial_{x\nu}\bar V$. 
 
We claim that $\bar U:=\partial_{x}\bar V$ is the decoupling field of the stochastic Hamiltonian system \eqref{H}. Let
\begin{equation}\label{eq:barxxx}
\left\{
\begin{aligned}
d\bar x_t^{*,t_0,\xi}=&~[A\bar x_t^{*,t_0,\xi}-B^2R^{-1}\bar U(t,\bar x_t^{*,t_0,\xi},\bar \nu_t^{*,t_0,\xi})-Bh(\bar \mu_t^{*,t_0,\xi})+f(\bar\nu_t^{*,t_0,\xi})+b(\bar\mu_t^{*,t_0,\xi})]dt\\
&+\sigma dW_t+\sigma_0dW^0_t,\\
\bar x_{t_0}^{*,t_0,\xi}=&~\xi
\end{aligned}
\right.
\end{equation}
 where
\[
\bar\mu_t^{*,t_0,\xi}:=\rho\big(-R^{-1}B\mathbb E\big[\bar U(t,\bar x_t^{*,t_0,\xi},\nu_t^{*,t_0,\xi})|\mathcal{F}_t^{W^0}\big]\big)\quad\text{and}\quad \bar\nu_t^{*,t_0,\xi}:=\mathbb E\big[\bar x_t^{*,t_0,\xi}|\mathcal{F}_t^{W^0}\big].
\]
We then let 
\begin{equation}\label{eq:baryzz}
\begin{aligned}
&\bar y_t^{*,t_0,\xi}:=\bar U(t,\bar x_t^{*,t_0,\xi}, \bar \nu_t^{*,t_0,\xi}),\quad \bar z_t^{*,t_0,\xi}:=\sigma\partial_x\bar U(t,\bar x_t^{*,t_0,\xi},\bar \nu_t^{*,t_0,\xi}),\\
& \bar z_t^{0,*,t_0,\xi}:=\sigma_0[\partial_x\bar U(t,\bar x_t^{*,t_0,\xi},\bar\nu_t^{*,t_0,\xi})+\partial_\nu U(t,\bar x_t^{*,t_0,\xi},\bar\nu_t^{*,t_0,\xi})]. 
\end{aligned}
\end{equation}
It can be verified that $(\bar x^{*,t_0,\xi},\bar y^{*,t_0,\xi},\bar z^{*,t_0,\xi},\bar z^{0,*,t_0,\xi})$ is a strong solution to the stochastic Hamiltonian system \eqref{H} given $(\bar \mu^{*,t_0,\xi}, \bar \nu^{*,t_0,\xi})$. We cannot directly utilize the It\^o's formula and then use the vectorial master equation \eqref{eq:vecmaster} to obtain the verification since $\bar U$ is not regular enough. However, a perturbation argument can be applied to overcome the difficulty and we sketch the proof in the following. We first use the standard smooth mollifier to mollify the function $\bar U$ in $x$ and $\nu$ variables to obtain a sequence of smooth functions $\{\bar U_n\}_n$ with uniformly bounded $\partial_x \bar U_n$ and $\partial_{\nu} \bar U_n$. Since $\bar V$ satisfies the master equation \eqref{master}, we can show $\bar U_n$ satisfies a perturbed vectorial master equation. We can then define $\bar x_n^{*,t_0,\xi}$ and $(\bar y_n^{*,t_0,\xi}, \bar z_n^{*,t_0,\xi},\bar z_n^{0,*,t_0,\xi})$ as the ones in \eqref{eq:barxxx} and \eqref{eq:baryzz}, respectively, however using $\bar U_n$ instead of $\bar U$. It can be shown that $(\bar x_n^{*,t_0,\xi},\bar y_n^{*,t_0,\xi}, \bar z_n^{*,t_0,\xi},\bar z_n^{0,*,t_0,\xi})$ is a strong solution to a perturbed stochastic Hamiltonian system. By the properties of the sequence $\{\bar U_n\}_n$, we are able to show that $(\bar x_n^{*,t_0,\xi},\bar y_n^{*,t_0,\xi}, \bar z_n^{*,t_0,\xi},\bar z_n^{0,*,t_0,\xi})$ converges to $(\bar x^{*,t_0,\xi},\bar y^{*,t_0,\xi}, \bar z^{*,t_0,\xi},\bar z^{0,*,t_0,\xi})$ and, moreover, $(\bar x^{*,t_0,\xi},\bar y^{*,t_0,\xi}, \bar z^{*,t_0,\xi},\bar z^{0,*,t_0,\xi})$ is the strong solution to the stochastic Hamiltonian system \eqref{H} given $(\bar \mu^{*,t_0,\xi}, \bar \nu^{*,t_0,\xi})$. By the unique solvability of the stochastic Hamiltonian system \eqref{H}, we derive that $U(t,x,\nu)=\bar U(t,x,\nu)$ and thus $\partial_xV(t,x,\nu)=\partial_x\bar V(t,x,\nu)$ for any $(t,x,\nu)\in [0,T]\times\mathbb R\times\mathbb R$. Therefore, there exists a function $c:[0,T]\times\mathbb R\to\mathbb R$ such that $V(t,x,\nu)=\bar V(t,x,\nu)+c(t,\nu)$ for any $(t,x,\nu)\in [0,T]\times\mathbb R\times\mathbb R$. Since both $V$ and $\bar V$ satisfy the master equation \eqref{master}, we can check that $c$ satisfies
   \begin{equation}\label{c}
 \left\{
 \begin{aligned}
&\partial_tc(t,\nu) +\partial_{\nu}c(t,\nu)A\nu+\frac{1}{2}\partial_{\nu\nu}c(t,\nu)\sigma_0^2=0,\\
 &c(T,\nu)=0.
 \end{aligned}
 \right.
 \end{equation}
 Then, by the comparison principle (see e.g. \cite{CIL1994}), we have $c(t,\nu)=0$ for any $(t,\nu)\in [0,T]\times\mathbb R$. Therefore $V(t,x,\nu)=\bar V(t,x,\nu)$ for any $(t,x,\nu)\in [0,T]\times\mathbb R\times\mathbb R$. Therefore, the uniqueness result for the master equation \eqref{master} follows.
\end{proof}

\section{$N$-player game}

Following the same idea in the Problem (MF), we shall first characterize the optimal control for each player in the Problem (NP). In this section, we need $f=b=0$ and $\sigma>0$. 



Suppose that $\boldsymbol{\alpha}^*:=(\alpha^{*,1},\ldots,\alpha^{*,N})$ is the Nash equilibrium (NE) for the Problem (NP). Thus, for any i.i.d. $\xi^i\in L_{\widetilde{\mathcal{F}}_0}^2$, $1\leq i\leq N$, the optimal path of the $i$th player is given by
 \begin{equation}
 \left\{
 \begin{aligned}\label{x*i}
dx^{*,i}_t=&~[Ax^{*,i}_t+B\alpha^{*,i}_t]dt+\sigma dW^i_t+\sigma_0dW^0_t,\\
 x^{*,i}_0=&~\xi^i.
 \end{aligned}
 \right.
 \end{equation}
We then apply the stochastic maximum principle to the Problem (NP) to obtain the following optimality condition: 
\begin{equation}\label{nopen}
\boldsymbol{\alpha}_t^*:=(\alpha_t^{*,i})_{1\leq i\leq N}=\left(-R^{-1}By_t^{*,i}-h(\mu^{N,i}_{\boldsymbol{\alpha}^*_t})\right)_{1\leq i\leq N}
\end{equation}
where $(\boldsymbol{y}^{*},\boldsymbol{z}^{*},\boldsymbol{z}^{0,*}):=((y^{*,i})_{1\leq i\leq N},(z^{*,i,j})_{1\leq i,j\leq N}, (z^{0,*,i})_{1\leq i\leq N})$ solves the following BSDE
 \begin{equation}
 \left\{
 \begin{aligned}\label{y*i}
dy^{*,i}_t=&-[Ay^{*,i}_t+Qx^{*,i}_t+Ql(\nu^{N,i}_{\boldsymbol{x}^*_t})]dt+\sum_{j=1}^Nz_t^{*,i,j} dW^j_t+z^{0,*,i}dW^0_t,\\
 y^{*,i}_T=& G(x_T^{*,i}+g(\nu^{N,i}_{\boldsymbol{x}^*_T})).
 \end{aligned}
 \right.
 \end{equation}
%
%
%
Recalling the definition of the $\mu^{N,i}_{\boldsymbol{\alpha}_{\cdot}^*}$, we have for any $1\leq i\leq N$
\begin{equation}\label{muNi}
\mu^{N,i}_{\boldsymbol{\alpha}^*_t}+\frac{1}{N-1}\underset{j\neq i}{\sum}h(\mu^{N,j}_{\boldsymbol{\alpha}^*_t})=\Delta_{*,t}^{N,i}:=-R^{-1}B\lambda^{N,i}_{\boldsymbol{y}_t^{*}},
\end{equation}
where
\begin{equation}\label{Deltai}
\lambda^{N,i}_{\boldsymbol{y}}:=\frac{1}{N-1}\sum_{j\not= i}y^j\quad\text{for any $\boldsymbol{y}=(y^1,\ldots,y^N)$}. 
\end{equation}

Introducing 
\[h_i^N(\mu^{N,1}_{\boldsymbol{\alpha}^*_t},\ldots,\mu^{N,N}_{\boldsymbol{\alpha}^*_t}):=\frac{1}{N-1}\underset{j\neq i}{\sum}h(\mu^{N,j}_{\boldsymbol{\alpha}^*_t}), \]
then the above relationship \eqref{muNi} is equivalent to the following $N$-coupled equations:\begin{equation}
\begin{array}{l}
 \mu^{N,1}_{\boldsymbol{\alpha}^*_t}+h_1^N(\mu^{N,1}_{\boldsymbol{\alpha}^*_t},\ldots,\mu^{N,N}_{\boldsymbol{\alpha}^*_t})=\Delta^{N,1}_{*,t},\\
\quad\vdots \qquad\qquad\qquad ~~\vdots\qquad \qquad~~~~\vdots\\
\mu^{N,N}_{\boldsymbol{\alpha}^*_t}+h_N^N(\mu^{N,1}_{\boldsymbol{\alpha}^*_t},\ldots,\mu^{N,N}_{\boldsymbol{\alpha}^*_t})=\Delta^{N,N}_{*,t},
\end{array}
\end{equation}
which can also be written as 
\begin{equation*}
\begin{array}{l}
F_1(\Delta_{*,t}^{N,i},\ldots,\Delta_{*,t}^{N,N},\mu^{N,1}_{\boldsymbol{\alpha}^*_t},\ldots,\mu^{N,N}_{\boldsymbol{\alpha}^*_t})=0,\\
\qquad\qquad\qquad~~\vdots\qquad\qquad \qquad\qquad\vdots\\
F_N(\Delta_{*,t}^{N,1},\ldots,\Delta_{*,t}^{N,N},\mu^{N,1}_{\boldsymbol{\alpha}^*_t},\ldots,\mu^{N,N}_{\boldsymbol{\alpha}^*_t})=0,
\end{array}
\end{equation*}
where
\[F_i(\Delta^{1},\ldots,\Delta^{N},\mu^{1},\ldots,\mu^{N}):=  \mu^{i}+h_i^N(\mu^{1},\ldots,\mu^{N})-\Delta^{i}.\]
Then
\begin{equation}\label{eq:DF}
\begin{aligned}
D_{\mu}F_t:=&\left[\begin{array}{cccc} 
\frac{\partial F_1}{\partial {\mu^{1}}} &    \ldots  & \frac{\partial F_1}{\partial{\mu^{N}}} \\ 
 \vdots&      & \vdots\\ 
\frac{\partial F_N}{\partial {\mu^{1}}} & \ldots & \frac{\partial F_N}{\partial {\mu^{N}}} 
\end{array}\right] (\Delta_{*,t}^{N,1},\ldots,\Delta_{*,t}^{N,N},\mu^{N,1}_{\boldsymbol{\alpha}^*_t},\ldots,\mu^{N,N}_{\boldsymbol{\alpha}^*_t})\\
=&\left[\begin{array}{cccc} 
1& \frac{1}{N-1}h^\prime(\mu^{N,2}_{\boldsymbol{\alpha}^*_t})&\ldots& \frac{1}{N-1}h^\prime(\mu^{N,N}_{\boldsymbol{\alpha}^*_t})  \\ 
\frac{1}{N-1}h^\prime(\mu^{N,1}_{\boldsymbol{\alpha}^*_t})&1&\ldots&\frac{1}{N-1}h^\prime(\mu^{N,N}_{\boldsymbol{\alpha}^*_t}) \\
 \vdots&  \vdots &   & \vdots\\ 
\frac{1}{N-1}h^\prime(\mu^{N,1}_{\boldsymbol{\alpha}^*_t})&\frac{1}{N-1}h^\prime(\mu^{N,2}_{\boldsymbol{\alpha}^*_t})& \ldots &1
\end{array}\right].
\end{aligned}
\end{equation}

We need the following assumption to show the above matrix $D_\mu F_t$ is uniformly bounded away from zero, uniformly in $N$.

\textbf{Assumption (B')} There exists some $\varepsilon_0>0$ such that $|h^{\prime}(\cdot)|\leq 1-\varepsilon_0$.

\begin{lemma}\label{lem:implicit}
Under Assumptions (B') and (C), for any positive integer $N$ we have the eigenvalues $\{\lambda_{t}^{N,j}\}_{1\leq j\leq N}$ of the matrix $D_{\mu}F_t$ defined in \eqref{eq:DF} satisfying
\begin{equation}\label{eq:Nepsilon0}
\min_{1\leq j\leq N}\lambda_{t}^{N,j}\geq \varepsilon_0\quad\text{for all $t\in [0,T]$}.
\end{equation}
\end{lemma}
\begin{proof}
We use the Gershgorin circle theorem to derive
\[
\{\lambda_{t}^{N,j}\}_{1\leq j\leq N}\subset \Big\{\lambda\in\mathbb C\,:\,\left|\lambda-1\right|\leq \frac{1}{N-1}\sum_{j\not =i} h'(\mu^{N,j}_{\boldsymbol{\alpha}^*_t}) \Big\}.
\]
Applying Assumptions (B') and (C), for any positive integer $N$, we have
\[
\{\lambda_{t}^{N,j}\}_{1\leq j\leq N}\subset \left\{\lambda\in\mathbb C\,:\,\left|\lambda-1\right|\leq 1-\varepsilon_0 \right\}
\]
and thus \eqref{eq:Nepsilon0} holds.
\end{proof}

With Lemma \ref{lem:implicit}, we apply the implicit function theorem to obtain that there exist uniformly Lipschitz continuous functions $\rho_i^N:\mathbb R^d\to\mathbb R$, $1\leq i \leq N$, with their Lipschitz constants independent of $N$ such that 
\begin{equation}\label{rhoiN}
\mu^{N,i}_{\boldsymbol{\alpha}^*_t}=\rho_i^N(\Delta_{*,t}^{N,1},\ldots,\Delta_{*,t}^{N,N}), 
\end{equation}
and thus we deduce the NE
\begin{equation}\label{Nopen}
\boldsymbol{\alpha}_t^*=(\alpha_t^{*,i})_{1\leq i\leq N}=\left(-R^{-1}By_t^{*,i}-h\big(\rho_i^N(\Delta^{N,1}_{*,t},\ldots,\Delta^{N,N}_{*,t})\big)\right)_{1\leq i\leq N}.
\end{equation}
Substituting the NE into the state equation \eqref{x*i}, we obtain 
\begin{equation}\label{NFSDE}
\left\{
\begin{aligned}
dx_t^{*,i}=&~\Big[Ax_t^{*,i}-B^2R^{-1}y_t^{*,i}-Bh\big(\rho_i^N(\Delta^{N,1}_{*,t},\ldots,\Delta^{N,N}_{*,t})\big)\Big]dt+\sigma dW_t^i+\sigma_0dW^0_t,\\
x^{*,i}_0=&~\xi^i.
\end{aligned}
\right.
\end{equation}
Combining \eqref{NFSDE} and \eqref{y*i}, we derive the following system of $N$-coupled FBSDEs:
\begin{equation}\label{NFBSDE}
\left\{
\begin{aligned}
dx_t^{*,i}=&~\Big[Ax_t^{*,i}-B^2R^{-1}y_t^{*,i}-Bh\big(\rho_i^N(\Delta^{N,1}_{*,t},\ldots,\Delta^{N,N}_{*,t})\big)\Big]dt+\sigma dW_t^i+\sigma_0dW^0_t,\\
dy_t^{*,i}=&-[Ay_t^{*,i}+Qx_t^{*,i}+Ql(\nu^{N,i}_{\boldsymbol{x}_t^*})]dt+\underset{j=1}{\overset{N}{\sum}}z_t^{*,i,j}dW_t^j+z_t^{0,*,i}dW^0_t,\\
x^{*,i}_0=&~\xi^i,~y^{*,i}_T=G\big(x^{*,i}_T+g(\nu^{N,i}_{\boldsymbol{x}_T^*})\big).
\end{aligned}
\right.
\end{equation}

We show in the following theorem that, under Assumptions (A), (B'), (C), the above system \eqref{NFBSDE} is well-posed.

\begin{theorem}\label{thm:SMP}
Suppose that Assumptions (A), (B'), (C) hold and $\sigma>0$. Let $\xi^i\in L_{\widetilde{\mathcal{F}}_0}^2$, $1\leq i\leq N$, be i.i.d. random variables.\\
(i) The following system of $N$-coupled FBSDEs
\begin{equation}
\left\{
\begin{aligned}\label{NNCE}
dx_t^{*,i}=&~\Big[(A-B^2R^{-1}P_t)x_t^{*,i}-B^2R^{-1}\varphi_t^{*,i}-Bh\Big(k_i^N(t,\nu^{N,1}_{\boldsymbol{x}^*_t},\ldots,\nu^{N,N}_{\boldsymbol{x}^*_t},\lambda^{N,1}_{\boldsymbol{\varphi}^*_t},\ldots,\lambda^{N,N}_{\boldsymbol{\varphi}^*_t})\Big)\Big]dt\\
&+\sigma dW_t^i+\sigma_0 dW^0_t,\\
d\varphi_t^{*,i}=&-\Big[(A-B^2R^{-1}P_t)\varphi_t^{*,i}+Ql(\nu^{N,i}_{\boldsymbol{x}^*_t})-P_tBh\Big(k_i^N(t,\nu^{N,1}_{\boldsymbol{x}^*_t},\ldots,\nu^{N,N}_{\boldsymbol{x}^*_t},\lambda^{N,1}_{\boldsymbol{\varphi}^*_t},\ldots,\lambda^{N,N}_{\boldsymbol{\varphi}^*_t})\Big)\Big]dt\\
&+\underset{j=1}{\overset{N}{\sum}}\Lambda^{*,i,j}_tdW_t^j+\Lambda_t^{0,*,i}dW^0_t,\\
\
x^{*,i}_0=&~\xi^i, \varphi^{*,i}_T=Gg(\nu^{N,i}_{\boldsymbol{x}^*_T}),
\end{aligned}
\right.
\end{equation}
admits a unique strong solution \[(\boldsymbol{x}^{*}:=(x^{*,i})_{1\leq i\leq N},\boldsymbol{\varphi}^{*}:=(\varphi^{*,i})_{1\leq i\leq N}, \boldsymbol{\Lambda}^{*}:=(\Lambda^{*,i,j})_{1\leq i,j\leq N},\boldsymbol{\Lambda}^{0,*}:=(\Lambda^{0,*,i})_{1\leq i\leq N}), \]
where $P$ is the unique solution to \eqref{P} and $k_i^N$, $1\leq i\leq N$, is defined by: for $\boldsymbol{x},\boldsymbol{\varphi}\in \mathbb R^N$
\begin{equation}\label{eq:kiN}
k_i^N(t,\nu^{N,1}_{\boldsymbol{x}},\ldots,\nu^{N,N}_{\boldsymbol{x}},\lambda^{N,1}_{\boldsymbol{\varphi}},\ldots,\lambda^{N,N}_{\boldsymbol{\varphi}}):=\rho_i^N\big(-R^{-1}B[P_t\nu^{N,1}_{\boldsymbol{x}}+\lambda^{N,1}_{\boldsymbol{\varphi}}],\ldots,-R^{-1}B[P_t\nu^{N,N}_{\boldsymbol{x}}+\lambda^{N,N}_{\boldsymbol{\varphi}}]\big).
\end{equation}
(ii) Given $\boldsymbol{x}^{*}$ in (i), we define 
\begin{equation}
\left.
\begin{aligned}\label{eq:x*y*z*}
& \qquad \qquad \qquad \qquad \qquad \boldsymbol{y}_t^{*}:=(y_t^{*,i})_{1\leq i\leq N}:=(P_t x_t^{*,i}+\varphi_t^{*,i})_{1\leq i\leq N},\\
&\boldsymbol{z}_t^{*}:=(z_t^{*,i,j})_{1\leq i,j\leq N}:=(P_t\sigma+\Lambda_t^{*,i,j})_{1\leq i,j\leq N},\quad \boldsymbol{z}_t^{0,*}:=(z_t^{0,*,i})_{1\leq i\leq N}:=(P_t\sigma_0+\Lambda_t^{0,*,i})_{1\leq i\leq N}.
\end{aligned}
\right.
 \end{equation}
  Then the system \eqref{NFBSDE} of $N$-coupled FBSDEs admits a unique strong solution $(\boldsymbol{x}^{*},\boldsymbol{y}^{*},\boldsymbol{z}^{*}, \boldsymbol{z}^{0,*})$.\\
(iii) The NE of the Problem (NP) can be represented in the feedback form
\begin{equation}\label{nclose}
\boldsymbol{\alpha}_t^*:=(\alpha_t^{*,i})_{1\leq i\leq N}=\left(-R^{-1}B(P_tx_t^{*,i}+\varphi_t^{*,i})-h\Big(k_i^N(t,\nu^{N,1}_{\boldsymbol{x}^*_t},\ldots,\nu^{N,N}_{\boldsymbol{x}^*_t},\lambda^{N,1}_{\boldsymbol{\varphi}^*_t},\ldots,\lambda^{N,N}_{\boldsymbol{\varphi}^*_t})\Big)\right)_{1\leq i\leq N}.
\end{equation}

\end{theorem}
\begin{proof}
(i) The proof is similar to that of the well-posedness of the NCE system \eqref{NCE} for the mean field game. We remind that $\eta_t=\exp\int_t^T(A-B^2R^{-1}P_s)ds$ and
\[
\varepsilon_1\leq\eta_t\leq\frac{1}{\varepsilon_1},\quad\text{for some $\varepsilon_1>0$.}
\]
For any $t\in[0,T]$, we introduce the following transformations
\[
\left.
\begin{aligned}
\boldsymbol{\widetilde{x}}_t^{*}:=(\widetilde{x}_t^{*,i})_{1\leq i\leq N}:=(\eta_tx_t^{*,i})_{1\leq i\leq N}, \quad \boldsymbol{\widetilde{\varphi}}^{*}_t:=(\widetilde{\varphi}_t^{*,i})_{1\leq i\leq N}:=(\eta_t^{-1}\varphi_t^{*,i})_{1\leq i\leq N},\qquad\ \\
\boldsymbol{\widetilde{\Lambda}}_t^{*}:=(\widetilde{\Lambda}_t^{*,i,j})_{1\leq i,j\leq N}:=(\eta_t^{-1}\Lambda_t^
{*,i,j})_{1\leq i,j\leq N},\quad \boldsymbol{\widetilde{\Lambda}}^{0,*}_t:=(\widetilde{\Lambda}_t^{0,*,i})_{1\leq i\leq N}:=(\eta_t^{-1}\Lambda_t^{0,*,i})_{1\leq i\leq N},
\end{aligned}
\right.
 \]
and we define for any $\boldsymbol{x},\boldsymbol{y}\in\mathbb R^N$
\[\widetilde{k}_i^N(t,\nu^{N,1}_{\boldsymbol{x}},\ldots,\nu^{N,N}_{\boldsymbol{x}},\lambda^{N,1}_{\boldsymbol{y}},\ldots,\lambda^{N,N}_{\boldsymbol{y}}):=\rho_i\big(-R^{-1}B[P_t\eta_t^{-1}\nu^{N,1}_{\boldsymbol{x}}+\eta_t\lambda^{N,1}_{\boldsymbol{y}}],\ldots,-R^{-1}B[P_t\eta_t^{-1}\nu^{N,N}_{\boldsymbol{x}}+\eta_t\lambda^{N,N}_{\boldsymbol{y}}]\big).\]
By a straightforward calculation, we can verify that $(\boldsymbol{\widetilde{x}}^{*}, \boldsymbol{\widetilde{\varphi}}^{*}, \boldsymbol{\widetilde{\Lambda}}^{*},\boldsymbol{\widetilde{\Lambda}}^{0,*})$ corresponds to the following system of $N$-coupled FBSDEs:
\begin{equation}
\left\{
\begin{aligned}\label{tranNNCE}
d\widetilde{x}_t^{*,i}=&~\Big[-\eta_t^2B^2R^{-1}\widetilde{\varphi}_t^{*,i}-\eta_tBh\Big(\widetilde{k}_i^N(t,\nu^{N,1}_{\boldsymbol{\widetilde{x}^*_t}},\ldots,\nu^{N,N}_{\boldsymbol{\widetilde{x}^*_t}},\lambda^{N,1}_{\boldsymbol{\widetilde{\varphi}^*_t}},\ldots,\lambda^{N,N}_{\boldsymbol{\widetilde{\varphi}^*_t}})\Big)\Big]dt\\
&+\eta_t\sigma dW_t^i+\eta_t\sigma_0 dW^0_t,\\
d\widetilde{\varphi}_t^{*,i}=&-\Big[\eta_t^{-1}Ql(\eta_t^{-1}\nu^{N,i}_{\boldsymbol{\widetilde{x}}^*_t})-\eta_t^{-1}P_tBh\Big(\widetilde{k}^N_i(t,\nu^{N,1}_{\boldsymbol{\widetilde{x}^*_t}},\ldots,\nu^{N,N}_{\boldsymbol{\widetilde{x}^*_t}},\lambda^{N,1}_{\boldsymbol{\widetilde{\varphi}_t^*}},\ldots,\lambda^{N,N}_{\boldsymbol{\widetilde{\varphi}_t^*}})\Big)\Big]dt\\
&+\underset{j=1}{\overset{N}{\sum}}\widetilde{\Lambda}^{*,i,j}_tdW_t^j+\widetilde{\Lambda}^{0,*,i}_tdW^0_t,\\
\
\widetilde{x}^{*,i}_0=&~\eta_0\xi^i, \widetilde{\varphi}^{*,i}_T=Gg(\nu^{N,i}_{\boldsymbol{\widetilde{x}}^*_T}),
\end{aligned}
\right.
\end{equation}
and the system \eqref{NNCE} is equivalent to the system \eqref{tranNNCE}.
Since $\sigma,\sigma_0>0$, the well-posedness of the system \eqref{tranNNCE} again can be guaranteed by \cite[Theorem 2.6]{Delarue2002}. Therefore we derive the well-posedness of the system \eqref{NNCE}.

(ii) Let $(\boldsymbol{y}^{*},\boldsymbol{z}^{*}, \boldsymbol{z}^{0,*})$ be as in \eqref{eq:x*y*z*}. Then we can easily check that $(\boldsymbol{x}^{*},\boldsymbol{y}^{*},\boldsymbol{z}^{*}, \boldsymbol{z}^{0,*})$ is a strong solution to the system \eqref{NFBSDE}. The uniqueness of the strong solution to \eqref{NFBSDE} on $[0,T]$ follows from the standard local well-posedness theory of FBSDEs.

(iii) Using \eqref{nopen} and the $\boldsymbol{y}^{*}$ given in \eqref{eq:x*y*z*}, we know that $\alpha^{*}$ in \eqref{nclose} is the NE of the Problem (NP) in the feedback form.
\end{proof}

\begin{remark}
In Theorem \ref{thm:SMP}-(iii), we show the N-player open-loop Nash equilibrium is in fact of closed-loop form. However, we emphasize that such open-loop Nash equilibrium is not an $N$-player closed-loop Nash equilibrium. 
\end{remark}

We introduce the following two PDEs. Their solutions $u^{N,i}(t,\boldsymbol{x})$ and $\Psi^{N,i}(t,\boldsymbol{x})$ serve as decoupling fields of the FBSDEs \eqref{NFBSDE} and \eqref{NNCE}, respectively.
\begin{equation}
\left\{
\begin{aligned} \label{eq:vecnashsystem}
&\partial_tu^{N,i}(t,\boldsymbol{x})+\underset{j=1}{\overset{N}{\sum}}\partial_{x^j}u^{N,i}(t,\boldsymbol{x})\Big[Ax^j-B^2R^{-1}u^{N,j}(t,\boldsymbol{x})-Bh\Big(k_j^N(t,\nu^{N,1}_{\boldsymbol{x}},\ldots,\nu^{N,N}_{\boldsymbol{x}},\lambda^{N,1}_{\boldsymbol{\Psi}(t,\boldsymbol{x})},\ldots,\lambda^{N,N}_{\boldsymbol{\Psi}(t,\boldsymbol{x})})\Big)\Big]\\
&+\frac{1}{2}\underset{j,\tau=1}{\overset{N}{\sum}}\partial_{x^j}\partial_{x^\tau}u^{N,i}(t,\boldsymbol{x})\sigma_0^2+\frac{1}{2}\underset{j=1}{\overset{N}{\sum}}\partial_{x^jx^j}u^{N,i}(t,\boldsymbol{x})\sigma^2+Au^{N,i}(t,\boldsymbol{x})+Qx^i+Ql(\nu^{N,i}_{\boldsymbol{x}})=0,\\
&u^{N,i}(T,\boldsymbol{x})=G\big(x^i+g(\nu^{N,i}_{\boldsymbol{x}})\big),
\end{aligned}
\right.
\end{equation}
and 

\begin{equation}
\left\{
\begin{aligned} \label{Psi}
&\partial_t\Psi^{N,i}(t,\boldsymbol{x})+\underset{j=1}{\overset{N}{\sum}}\partial_{x^j}\Psi^{N,i}(t,\boldsymbol{x})\Big[(A-B^2R^{-1}P_t)x^j-B^2R^{-1}\Psi^{N,j}(t,\boldsymbol{x})\\
&-Bh\Big(k_j^N(t,\nu^{N,1}_{\boldsymbol{x}},\ldots,\nu^{N,N}_{\boldsymbol{x}},\lambda^{N,1}_{\boldsymbol{\Psi}(t,\boldsymbol{x})},\ldots,\lambda^{N,N}_{\boldsymbol{\Psi}(t,\boldsymbol{x})})\Big)\Big]+\frac{1}{2}\underset{j,\tau=1}{\overset{N}{\sum}}\partial_{x^j}\partial_{x^\tau}\Psi^{N,i}(t,\boldsymbol{x})\sigma_0^2\\
&+\frac{1}{2}\underset{j=1}{\overset{N}{\sum}}\partial_{x^jx^j}\Psi^{N,i}(t,\boldsymbol{x})\sigma^2+(A-B^2R^{-1}P_t)\Psi^{N,i}(t,\boldsymbol{x})+Ql(\nu^{N,i}_{\boldsymbol{x}})\\
&-P_tBh\Big(k_i^N(t,\nu^{N,1}_{\boldsymbol{x}},\ldots,\nu^{N,N}_{\boldsymbol{x}},\lambda^{N,1}_{\boldsymbol{\Psi}(t,\boldsymbol{x})},\ldots,\lambda^{N,N}_{\boldsymbol{\Psi}(t,\boldsymbol{x})}) \Big)=0,\\
&\Psi^{N,i}(T,\boldsymbol{x})=Gg(\nu^{N,i}_{\boldsymbol{x}}),
\end{aligned}
\right.
\end{equation}
with given $k_i^N$, $1\leq i\leq N$, defined in \eqref{eq:kiN}.

The following two theorems are the well-posedness results for \eqref{eq:vecnashsystem} and \eqref{Psi}. We omit their proofs since they are very similar to the ones of Theorems \ref{thm:Phi} and \ref{thm:vecmaster}.
\begin{theorem}\label{thm:Psi}
Suppose that Assumptions (A), (B'), (C) hold and $\sigma>0$. Then the PDE \eqref{Psi} admits a unique classical solution $\Psi^{N,i}\in C^{1,2}([0,T]\times\mathbb R^N)$, $1\leq i\leq N$, with bounded 1st and 2nd order derivatives.
\end{theorem}

\begin{theorem}\label{thm:vecnashsystem}
Suppose that Assumptions (A), (B'), (C) hold and $\sigma>0$. Then $u^{N,i}(t,\boldsymbol{x}):=P_tx^i+\Psi^{N,i}(t,\boldsymbol{x})\in C^{1,2}([0,T]\times\mathbb R^N)$, $1\leq i\leq N$, is the unique classical solution to \eqref{eq:vecnashsystem} with bounded $\partial_{x^j}u^{N,i}$ and $\partial_{x^kx^j}u^{N,i}$, $1\leq i,j,k\leq N$, where $\Psi^{N,i}(t,\boldsymbol{x})$, $1\leq i\leq N$, are given in Theorem \ref{thm:Psi}.
\end{theorem}
\begin{remark}
In general, we do not have a uniform bound, uniformly in $N$, for the derivatives $\partial_{x^j}\Psi^{N,i}$, $\partial_{x^kx^j}\Psi^{N,i}$ $\partial_{x^j}u^{N,i}$ and $\partial_{x^kx^j}u^{N,i}$ for any $1\leq i,j,k\leq N$ in Theorems \ref{thm:Psi} and \ref{thm:vecnashsystem}.
\end{remark}

\section{Convergence}
In this section, we will focus on showing the convergence from the $N$-player game to the mean field game under the Assumptions (A), (B'), (C) with $f=b=0$ and $\sigma>0$. It is worthy to mention that, unlike the mean field game, the open-loop Nash equilibria are different from the closed-loop Nash equilibria in $N$-player game. It would be extremely hard for us to prove the convergence for the closed-loop Nash equilibria from the $N$-player game to the mean field game since the well-posedness of the Nash system for value function remains open with both running and terminal costs having quadratic growth. Instead, we are able to show the convergence for open-loop Nash equilibria via the $N$-coupled PDEs system \eqref{Psi} as well as their corresponding values. 

The main results of this section include two parts. In the first part, we investigate that the solution $\{\Psi^{N,i}\}_{1\leq i \leq N}$ to the $N$-coupled PDEs system \eqref{Psi} converges to the solution $\Phi$ to the PDE \eqref{Phi} in some suitable sense. It allows us to show the convergence for the open-loop Nash equilibrium from the $N$-player game to the mean field game. As a byproduct, we are able to show the convergence of its corresponding value. The key idea of proof is that the suitable finite-dimensional projections of $\Phi$ is an approximate solution to the $N$-coupled PDEs system \eqref{Psi}. In the second part, we verify a propagation of chaos property for the associated optimal trajectories. 
 
 \subsection{Convergence of Nash equilibria}

For $1\leq i \leq N$, we introduce the finite dimensional projections of $\Phi$: for any $\boldsymbol{x}=(x^1,\ldots,x^N)\in\mathbb{R}^N$,
\begin{equation}\label{defu}
\phi^{N,i}(t,\boldsymbol{x})=\Phi(t,\nu^{N,i}_{\boldsymbol{x}})~~\text{with}~~\nu^{N,i}_{\boldsymbol{x}}=\frac{1}{N-1}\underset{j\neq i}{\sum} x^j,
\end{equation}
where $\Phi$ is the unique solution to the PDE \eqref{Phi}. In view of the regularity property of $\Phi$, we have the following smoothness result for $\phi^{N,i}$. Since the proof is rather easy, we omit it.
\begin{proposition}\label{upartial}
Suppose that Assumptions (A), (B), (C) hold and $\sigma>0$. For $1\leq i \leq N$ and $t\in [0,T]$, $\phi^{N,i}(t,\cdot)$ is $C^2(\mathbb R^N)$ with the following first and second order partial derivatives in $\boldsymbol{x}$: $\partial_{x^i}\phi^{N,i}(t,\boldsymbol{x})=0$;

when $j\neq i$
\begin{equation*}
\partial_{x^j}\phi^{N,i}(t,\boldsymbol{x})=\frac{1}{N-1}\partial_{\nu}\Phi(t,\nu^{N,i}_{\boldsymbol{x}}),\quad
\partial_{x^ix^j}\phi^{N,i}(t,\boldsymbol{x})=0,\quad
\partial_{x^jx^j}\phi^{N,i}(t,\boldsymbol{x})=\frac{1}{(N-1)^2}\partial_{\nu\nu}\Phi(t,\nu^{N,i}_{\boldsymbol{x}});
\end{equation*}

when $j\neq \tau\neq i$
\[\partial_{x^j}\partial_{x^\tau}\phi^{N,i}(t,\boldsymbol{x})=\frac{1}{(N-1)^2}\partial_{\nu\nu}\Phi(t,\nu^{N,i}_{\boldsymbol{x}}).\]
\end{proposition}
The following is a technical lemma needed for the convergence results later.
\begin{lemma}
Let Assumptions (A), (B'), (C) hold and $\sigma>0$. For any $\boldsymbol{x}=(x^1,\ldots,x^N)$ and $\boldsymbol{x}^\prime=(x^{\prime,1},\ldots,x^{\prime,N})\in\mathbb{R}^N$, it has 
\begin{itemize}
\item[\textup{(i)}] \begin{equation}\label{kiphik}
|k_i^N(t,\nu^{N,1}_{\boldsymbol{x}},\ldots,\nu^{N,N}_{\boldsymbol{x}},\lambda^{N,1}_{\boldsymbol{\phi}(t,\boldsymbol{x})},\ldots,\lambda^{N,N}_{\boldsymbol{\phi}(t,\boldsymbol{x})})-k\big(t,\nu^{N,i}_{\boldsymbol{x}},\Phi(t,\nu^{N,i}_{\boldsymbol{x}})\big)|=O\Big( \frac{1}{N}\Big(1+\frac{1}{N}\underset{j\neq i}\sum|x^i-x^j|\Big)     \Big);
\end{equation}
\item[\textup{(ii)}] \begin{equation}\label{kiphikpsi}
\begin{aligned}
&|k_i^N(t,\nu^{N,1}_{\boldsymbol{x}},\ldots,\nu^{N,N}_{\boldsymbol{x}},\lambda^{N,1}_{\boldsymbol{\phi}(t,\boldsymbol{x})},\ldots,\lambda^{N,N}_{\boldsymbol{\phi}(t,\boldsymbol{x})})-k_i^N(t,\nu^{N,1}_{\boldsymbol{x}},\ldots,\nu^{N,N}_{\boldsymbol{x}},\lambda^{N,1}_{\boldsymbol{\Psi}(t,\boldsymbol{x})},\ldots,\lambda^{N,N}_{\boldsymbol{\Psi}(t,\boldsymbol{x})})|\\
=&O\bigg(\frac{1}{N}\Big(1+\frac{1}{N}\underset{j\neq i}\sum|x^i-x^j|\Big)+\frac{1}{N}\underset{j\neq i}\sum|\phi^{N,j}(t,\boldsymbol{x})-\Psi^{N,j}(t,\boldsymbol{x})|\bigg);
\end{aligned}
\end{equation}
\item[\textup{(iii)}] \begin{equation}
\begin{aligned}\label{kinphixpsixprime}
&|k_i^N(t,\nu^{N,1}_{\boldsymbol{x}},\ldots,\nu^{N,N}_{\boldsymbol{x}},\lambda^{N,1}_{\boldsymbol{\phi}(t,\boldsymbol{x})},\ldots,\lambda^{N,N}_{\boldsymbol{\phi}(t,\boldsymbol{x})})-k_i^N(t,\nu^{N,1}_{\boldsymbol{x}^\prime},\ldots,\nu^{N,N}_{\boldsymbol{x}^\prime},\lambda^{N,1}_{\boldsymbol{\Psi}(t,\boldsymbol{x}^\prime)},\ldots,\lambda^{N,N}_{\boldsymbol{\Psi}(t,\boldsymbol{x}^\prime)}) |\\=&O\bigg(\frac{1}{N}\Big(1+\frac{1}{N}\underset{j\neq i}{\sum}|x^i-x^j|+\frac{1}{N}\underset{j\neq i}{\sum}|x^{\prime,i}-x^{\prime,j}|\Big)+\frac{1}{N}\underset{j\neq i}{\sum}|x^j-x^{\prime,j}|+\frac{1}{N}\underset{j\neq i}{\sum} |\phi^j(t,\boldsymbol{x}) -\Psi^j(t,\boldsymbol{x}^\prime)|\bigg),
\end{aligned}
\end{equation}
where $\Phi$ and $\{\Psi^{N,i}\}_{1\leq i\leq N}$ are solutions to equations \eqref{Phi} and \eqref{Psi} respectively.
\end{itemize} 
\end{lemma}
\begin{proof}
\begin{itemize}
\item[\textup{(i)}] We introduce 
\begin{equation}\label{eq:muNixphi}
\mu_{(\boldsymbol{x},\boldsymbol{\phi}(t,\boldsymbol{x}))}^{N,i}:=k_i^N(t,\nu^{N,1}_{\boldsymbol{x}},\ldots,\nu^{N,N}_{\boldsymbol{x}},\lambda^{N,1}_{\boldsymbol{\phi}(t,\boldsymbol{x})},\ldots,\lambda^{N,N}_{\boldsymbol{\phi}(t,\boldsymbol{x})}),\quad\text{and}\quad \mu^i:=k\big(t,\nu^{N,i}_{\boldsymbol{x}},\Phi(t,\nu^{N,i}_{\boldsymbol{x}})\big),
\end{equation}
by the definitions of $k_i^N$, for $1\leq i\leq N$, and $k$, we have the following two equations hold: (using $\phi^{N,j}(t,\boldsymbol{x})=\Phi(t,\nu^{N,j}_{\boldsymbol{x}})
$)
\begin{equation}\label{muxphi}
\mu_{(\boldsymbol{x},\boldsymbol{\phi}(t,\boldsymbol{x}))}^{N,i}+\frac{1}{N-1}\underset{j\neq i}{\sum}h\Big(\mu_{(\boldsymbol{x},\boldsymbol{\phi}(t,\boldsymbol{x}))}^{N,j} \Big)
=-R^{-1}B\Big[P_t\nu^{N,i}_{\boldsymbol{x}}+\frac{1}{N-1}\underset{j\neq i}{\sum}\Phi(t,\nu^{N,j}_{\boldsymbol{x}})\Big]
\end{equation}
and
\begin{equation}\label{mui}
\mu^i+h(\mu^i)=-R^{-1}B\Big[P_t\nu^{N,i}_{\boldsymbol{x}}+\Phi(t,\nu^{N,i}_{\boldsymbol{x}})\Big]. 
\end{equation}

Noting that $\mu_{(\boldsymbol{x},\boldsymbol{\phi}(t,\boldsymbol{x}))}^{N,i}$ can be rewritten as 
\begin{equation*}
\mu_{(\boldsymbol{x},\boldsymbol{\phi}(t,\boldsymbol{x}))}^{N,i}=\frac{1}{N-1}\underset{j\neq i}{\sum}\Big[-R^{-1}B\Big(P_tx^j+\Phi(t,\nu^{N,j}_{\boldsymbol{x}})\Big)-h\Big(\mu_{(\boldsymbol{x},\boldsymbol{\phi}(t,\boldsymbol{x}))}^{N,j} \Big)\Big],
\end{equation*}
thus by the boundedness of $\Phi$ and $h$, it implies
\begin{equation}\label{muijxphi}
|\mu_{(\boldsymbol{x},\boldsymbol{\phi}(t,\boldsymbol{x}))}^{N,j}-\mu_{(\boldsymbol{x},\boldsymbol{\phi}(t,\boldsymbol{x}))}^{N,i}|\leq \frac{C}{N-1}(1+|x^i-x^j|).
\end{equation}

From \eqref{muxphi}, \eqref{mui} and Assumption (B'), we have 
\begin{equation*}
\begin{aligned}
&|\mu_{(\boldsymbol{x},\boldsymbol{\phi}(t,\boldsymbol{x}))}^{N,i}-\mu^i|\\
\leq& |\frac{1}{N-1}\underset{j\neq i}{\sum}h\Big(\mu_{(\boldsymbol{x},\boldsymbol{\phi}(t,\boldsymbol{x}))}^{N,j} \Big)
-h(\mu^i)|+C|\frac{1}{N-1}\underset{j\neq i}{\sum}\Phi(t,\nu^{N,j}_{\boldsymbol{x}})-\Phi(t,\nu^{N,i}_{\boldsymbol{x}})|\\
\leq &\frac{1-\varepsilon_0}{N-1}\underset{j\neq i}{\sum}|\mu_{(\boldsymbol{x},\boldsymbol{\phi}(t,\boldsymbol{x}))}^{N,j}-\mu^i|+\frac{C}{N-1}\Big(\frac{1}{N-1}\underset{j\neq i}{\sum}|x^i-x^j|\Big)\\
\leq & (1-\varepsilon_0)|\mu_{(\boldsymbol{x},\boldsymbol{\phi}(t,\boldsymbol{x}))}^{N,i}-\mu^i|+ \frac{1}{N-1}\underset{j\neq i}{\sum}|\mu_{(\boldsymbol{x},\boldsymbol{\phi}(t,\boldsymbol{x}))}^{N,j}-\mu_{(\boldsymbol{x},\boldsymbol{\phi}(t,\boldsymbol{x}))}^{N,i}|+\frac{C}{N-1}\Big(\frac{1}{N-1}\underset{j\neq i}{\sum}|x^i-x^j|\Big).
\end{aligned}
\end{equation*}
By \eqref{muijxphi}, one can get 
\begin{equation*}
\begin{aligned}
|\mu_{(\boldsymbol{x},\boldsymbol{\phi}(t,\boldsymbol{x}))}^{N,i}-\mu^i|\leq &\frac{C}{N-1}\underset{j\neq i}{\sum}|\mu_{(\boldsymbol{x},\boldsymbol{\phi}(t,\boldsymbol{x}))}^{N,j}-\mu_{(\boldsymbol{x},\boldsymbol{\phi}(t,\boldsymbol{x}))}^{N,i}|+\frac{C}{N-1}\Big(\frac{1}{N-1}\underset{j\neq i}{\sum}|x^i-x^j|\Big)\\
\leq & \frac{C}{N-1}\Big(1+\frac{1}{N-1}\underset{j\neq i}{\sum}|x^i-x^j|\Big)=O\Big( \frac{1}{N}\Big(1+\frac{1}{N}\underset{j\not = i}{\sum}|x^i-x^j|\Big)     \Big),
\end{aligned}
\end{equation*}
which completes the proof of  \eqref{kiphik}. 

\item[\textup{(ii)}] We also introduce 
\begin{equation}\label{munipsix}
\mu_{(\boldsymbol{x},\boldsymbol{\Psi}(t,\boldsymbol{x}))}^{N,i}:=k_i^N(t,\nu^{N,1}_{\boldsymbol{x}},\ldots,\nu^{N,N}_{\boldsymbol{x}},\lambda^{N,1}_{\boldsymbol{\Psi}(t,\boldsymbol{x})},\ldots,\lambda^{N,N}_{\boldsymbol{\Psi}(t,\boldsymbol{x})}),
\end{equation}
which implies
\begin{equation}\label{psini}
\mu_{(\boldsymbol{x},\boldsymbol{\Psi}(t,\boldsymbol{x}))}^{N,i}+\frac{1}{N-1}\underset{j\neq i}{\sum}h\Big(\mu_{(\boldsymbol{x},\boldsymbol{\Psi}(t,\boldsymbol{x}))}^{N,j} \Big)
=-R^{-1}B\Big[P_t\nu^{N,i}_{\boldsymbol{x}}+\frac{1}{N-1}\underset{j\neq i}{\sum}\Psi^{N,j}(t,\boldsymbol{x})\Big].
\end{equation}

Moreover, we have that \eqref{muxphi} can be rewritten as 
\begin{equation}\label{phini}
\mu_{(\boldsymbol{x},\boldsymbol{\phi}(t,\boldsymbol{x}))}^{N,i}+\frac{1}{N-1}\underset{j\neq i}{\sum}h\Big(\mu_{(\boldsymbol{x},\boldsymbol{\phi}(t,\boldsymbol{x}))}^{N,j} \Big)
=-R^{-1}B\Big[P_t\nu^{N,i}_{\boldsymbol{x}}+\frac{1}{N-1}\underset{j\neq i}{\sum}\phi^{N,j}(t,\boldsymbol{x})\Big].
\end{equation}

By \eqref{psini},\eqref{phini} and Assumption (B'), we deduce
\begin{equation}
\begin{aligned}\label{(ii)proof}
&|\mu_{(\boldsymbol{x},\boldsymbol{\phi}(t,\boldsymbol{x}))}^{N,i}-\mu_{(\boldsymbol{x},\boldsymbol{\Psi}(t,\boldsymbol{x}))}^{N,i}|\\
\leq &\frac{1}{N-1}\underset{j\neq i}{\sum}|h\Big(\mu_{(\boldsymbol{x},\boldsymbol{\phi}(t,\boldsymbol{x}))}^{N,j} \Big)-h\Big(\mu_{(\boldsymbol{x},\boldsymbol{\Psi}(t,\boldsymbol{x}))}^{N,j} \Big)|+\frac{C}{N-1}\underset{j\neq i}{\sum}|\phi^{N,j}(t,\boldsymbol{x})-\Psi^{N,j}(t,\boldsymbol{x})|\\
\leq &\frac{1-\varepsilon_0}{N-1}\underset{j\neq i}{\sum}|\mu_{(\boldsymbol{x},\boldsymbol{\phi}(t,\boldsymbol{x}))}^{N,j}-\mu_{(\boldsymbol{x},\boldsymbol{\Psi}(t,\boldsymbol{x}))}^{N,j}|+\frac{C}{N-1}\underset{j\neq i}{\sum}|\phi^{N,j}(t,\boldsymbol{x})-\Psi^{N,j}(t,\boldsymbol{x})|\\
\leq &(1-\varepsilon_0)|\mu_{(\boldsymbol{x},\boldsymbol{\phi}(t,\boldsymbol{x}))}^{N,i}-\mu_{(\boldsymbol{x},\boldsymbol{\Psi}(t,\boldsymbol{x}))}^{N,i}|+\frac{1}{N-1}\underset{j\neq i}{\sum}|\mu_{(\boldsymbol{x},\boldsymbol{\phi}(t,\boldsymbol{x}))}^{N,j}-\mu_{(\boldsymbol{x},\boldsymbol{\phi}(t,\boldsymbol{x}))}^{N,i}|\\
&+\frac{1}{N-1}\underset{j\neq i}{\sum}|\mu_{(\boldsymbol{x},\boldsymbol{\Psi}(t,\boldsymbol{x}))}^{N,j}-\mu_{(\boldsymbol{x},\boldsymbol{\Psi}(t,\boldsymbol{x}))}^{N,i}|+\frac{C}{N-1}\underset{j\neq i}{\sum}|\phi^{N,j}(t,\boldsymbol{x})-\Psi^{N,j}(t,\boldsymbol{x})|.
\end{aligned}
\end{equation}

Recalling \eqref{muijxphi} and the boundedness of $\Psi^{N,i}$ and $h$, we can get 
\begin{equation}\label{psinjni}
|\mu_{(\boldsymbol{x},\boldsymbol{\Psi}(t,\boldsymbol{x}))}^{N,j}-\mu_{(\boldsymbol{x},\boldsymbol{\Psi}(t,\boldsymbol{x}))}^{N,i}|\leq \frac{C}{N-1}(1+|x^i-x^j|).
\end{equation}

Based on the above estimates, we obtain 
\begin{equation*}
|\mu_{(\boldsymbol{x},\boldsymbol{\phi}(t,\boldsymbol{x}))}^{N,i}-\mu_{(\boldsymbol{x},\boldsymbol{\Psi}(t,\boldsymbol{x}))}^{N,i}|\leq  \frac{C}{N-1}\Big(1+\frac{1}{N-1}\underset{j\neq i}{\sum}|x^i-x^j|\Big)+\frac{C}{N-1}\underset{j\neq i}{\sum}|\phi^{N,j}(t,\boldsymbol{x})-\Psi^{N,j}(t,\boldsymbol{x})|,
\end{equation*}
which is exactly \eqref{kiphikpsi}. 
\item[\textup{(iii)}]
For any $\boldsymbol{x}^\prime=(x^{\prime,1},\ldots,x^{\prime,N})\in\mathbb{R}^N$, we have $\mu_{(\boldsymbol{x}^\prime,\boldsymbol{\Psi}(t,\boldsymbol{x}^\prime))}^{N,i}$ satisfies \eqref{munipsix} and \eqref{psini} with $\boldsymbol{x}$ repalced by $\boldsymbol{x}^\prime$.
Similar to \eqref{(ii)proof}, we can obtain
\begin{equation*}
\begin{aligned}
|\mu_{(\boldsymbol{x},\boldsymbol{\phi}(t,\boldsymbol{x}))}^{N,i}-\mu_{(\boldsymbol{x}^\prime,\boldsymbol{\Psi}(t,\boldsymbol{x}^\prime))}^{N,i}|\leq &\frac{1}{N-1}\underset{j\neq i}{\sum}|\mu_{(\boldsymbol{x},\boldsymbol{\phi}(t,\boldsymbol{x}))}^{N,j}-\mu_{(\boldsymbol{x},\boldsymbol{\phi}(t,\boldsymbol{x}))}^{N,i}|+\frac{1}{N-1}\underset{j\neq i}{\sum}|\mu_{(\boldsymbol{x}^\prime,\boldsymbol{\Psi}(t,\boldsymbol{x}^\prime))}^{N,j}-\mu_{(\boldsymbol{x}^\prime,\boldsymbol{\Psi}(t,\boldsymbol{x}^\prime))}^{N,i}|\\
&+\frac{C}{N-1}\underset{j\neq i}{\sum}|x^j-x^{\prime,j}|+\frac{C}{N-1}\underset{j\neq i}{\sum}|\phi^{N,j}(t,\boldsymbol{x})-\Psi^{N,j}(t,\boldsymbol{x}^\prime)|.
\end{aligned}
\end{equation*}

Noting \eqref{muijxphi} and $|\mu_{(\boldsymbol{x},\boldsymbol{\Psi}(t,\boldsymbol{x}^\prime))}^{N,j}-\mu_{(\boldsymbol{x},\boldsymbol{\Psi}(t,\boldsymbol{x}^\prime))}^{N,i}|\leq \frac{C}{N-1}(1+|x^{\prime,i}-x^{\prime,j}|)$, we have 
\begin{equation*}
\begin{aligned}
|\mu_{(\boldsymbol{x},\boldsymbol{\phi}(t,\boldsymbol{x}))}^{N,i}-\mu_{(\boldsymbol{x}^\prime,\boldsymbol{\Psi}(t,\boldsymbol{x}^\prime))}^{N,i}|\leq &\frac{C}{N-1}\Big(1+\frac{1}{N-1}\underset{j\neq i}{\sum}|x^i-x^j|+\frac{1}{N-1}\underset{j\neq i}{\sum}|x^{\prime,i}-x^{\prime,j}|\Big)\\
&+\frac{C}{N-1}\underset{j\neq i}{\sum}|x^j-x^{\prime,j}|+\frac{C}{N-1}\underset{j\neq i}{\sum} |\phi^j(t,\boldsymbol{x}) -\Psi^j(t,\boldsymbol{x}^\prime)|,
\end{aligned}
\end{equation*}
which implies \eqref{kinphixpsixprime}. 
\end{itemize}
\end{proof}

In the following result, we are interested in the equation satisfied by $\phi^{N,i}(t,\boldsymbol{x})$, $1\leq i\leq N$. It can be shown that $\phi^{N,i}(t,\boldsymbol{x})$, $1\leq i\leq N$, is an approximate solution to the $N$-coupled PDEs system \eqref{Psi}.  

\begin{theorem}
Suppose that Assumptions (A), (B'), (C) hold. Then there exists a constant $C>0$ such that $\phi^{N,i}(t,\boldsymbol{x})$, $1\leq i\leq N$, satisfies
\begin{equation}\label{phi}
\left\{
\begin{aligned}
&\partial_t\phi^{N,i}(t,\boldsymbol{x})+\underset{j=1}{\overset{N}{\sum}}\partial_{x^j}\phi^{N,i}(t,\boldsymbol{x})\Big[(A-B^2R^{-1}P_t)x^j-B^2R^{-1}\phi^{N,j}(t,\boldsymbol{x})\\
&-Bh\Big(k_j^N(t,\nu^{N,1}_{\boldsymbol{x}},\ldots,\nu^{N,N}_{\boldsymbol{x}},\lambda^{N,1}_{\boldsymbol{\phi}(t,\boldsymbol{x})},\ldots,\lambda^{N,N}_{\boldsymbol{\phi}(t,\boldsymbol{x})})\Big)\Big]+\frac{1}{2}\underset{j,\tau=1}{\overset{N}{\sum}}\partial_{x^j}\partial_{x^\tau}\phi^{N,i}(t,\boldsymbol{x})\sigma_0^2\\
&+\frac{1}{2}\underset{j=1}{\overset{N}{\sum}}\partial_{x^jx^j}\phi^{N,i}(t,\boldsymbol{x})\sigma^2+(A-B^2R^{-1}P_t)\phi^{N,i}(t,\boldsymbol{x})+Ql(\nu^{N,i}_{\boldsymbol{x}})\\
&-P_tBh\Big(k_i^N(t,\nu^{N,1}_{\boldsymbol{x}},\ldots,\nu^{N,N}_{\boldsymbol{x}},\lambda^{N,1}_{\boldsymbol{\phi}(t,\boldsymbol{x})},\ldots,\lambda^{N,N}_{\boldsymbol{\phi}(t,\boldsymbol{x})})\Big)=r^{N,i}(t,\boldsymbol{x}),\\
&\phi^{N,i}(T,\boldsymbol{x})=Gg(\nu^{N,i}_{\boldsymbol{x}}),
\end{aligned}
\right.
\end{equation}
where 
\begin{equation}\label{rni}
|r^{N,i}(t,\boldsymbol{x})|\leq\frac{C}{N}\Big(1+\frac{1}{N}\underset{j=1}{\overset{N}\sum}|x^i-x^j|\Big).
\end{equation}
\end{theorem}
\begin{proof}
We evaluate the PDE \eqref{Phi} at $(t,\nu^{N,i}_{\boldsymbol{x}})\in [0,T]\times\mathbb R$:
\begin{equation}\label{wvalue}
 \begin{aligned}
&\partial_t\Phi(t,\nu^{N,i}_{\boldsymbol{x}})+\partial_{\nu}\Phi(t,\nu^{N,i}_{\boldsymbol{x}})\Big[(A-B^2R^{-1}P_t)\nu^{N,i}_{\boldsymbol{x}}-B^2R^{-1}\Phi(t,\nu^{N,i}_{\boldsymbol{x}})-Bh\Big(k\big(t,\nu^{N,i}_{\boldsymbol{x}},\Phi(t,\nu^{N,i}_{\boldsymbol{x}})\big)\Big)\Big]\\
 &
 +\frac{1}{2}\partial_{\nu\nu}\Phi(t,\nu^{N,i}_{\boldsymbol{x}})\sigma_0^2+(A-B^2R^{-1}P_t)\Phi(t,\nu^{N,i}_{\boldsymbol{x}})+Ql(\nu^{N,i}_{\boldsymbol{x}})-P_tBh\Big(k\big(t,\nu^{N,i}_{\boldsymbol{x}},\Phi(t,\nu^{N,i}_{\boldsymbol{x}})\big)\Big)=0,
 \end{aligned}
 \end{equation}
where $\Phi(t,\nu^{N,i}_{\boldsymbol{x}})$ is the solution of \eqref{Phi} with respect to $\nu^{N,i}_{\boldsymbol{x}}$.

We first observe that by Proposition \ref{upartial} 
\begin{equation*}
\begin{aligned}
&\frac{1}{2}\underset{j=1}{\overset{N}{\sum}}\partial_{x^jx^j}\phi^{N,i}(t,\boldsymbol{x})\sigma^2+\frac{1}{2}\underset{j,\tau=1}{\overset{N}{\sum}}\partial_{x^jx^\tau}\phi^{N,i}(t,\boldsymbol{x})\sigma_0^2\\
=&\frac{1}{2}\partial_{x^ix^i}\phi^{N,i}(t,\boldsymbol{x})\sigma^2+\frac{1}{2}\underset{j\neq i}{\sum}\partial_{x^jx^j}\phi^{N,i}(t,\boldsymbol{x})\sigma^2+\frac{1}{2}\partial_{x^ix^i}\phi^{N,i}(t,\boldsymbol{x})\sigma_0^2\\
&+\frac{1}{2}\underset{j\neq i}{\sum}\partial_{x^jx^j}\phi^{N,i}(t,\boldsymbol{x})\sigma_0^2+\underset{j\neq i}{\sum}\partial_{x^ix^j}\phi^{N,i}(t,\boldsymbol{x})\sigma_0^2+\frac{1}{2}\underset{j\neq \tau\neq i}{\sum}\partial_{x^j}\partial_{x^\tau}\phi^{N,i}(t,\boldsymbol{x})\sigma_0^2\\
=&\frac{1}{2}\partial_{x^ix^i}\phi^{N,i}(t,\boldsymbol{x})(\sigma^2+\sigma_0^2)+\frac{1}{2}\underset{j\neq i}{\sum}\partial_{x^jx^j}\phi^{N,i}(t,\boldsymbol{x})(\sigma^2+\sigma_0^2)\\
&+\underset{j\neq i}{\sum}\partial_{x^i}\partial_{x^j}\phi^{N,i}(t,\boldsymbol{x})\sigma_0^2+\frac{1}{2}\underset{j\neq \tau\neq i}{\sum}\partial_{x^j}\partial_{x^\tau}\phi^{N,i}(t,\boldsymbol{x})\sigma_0^2\\
=&\frac{1}{2}\frac{1}{N-1}\partial_{\nu\nu}\Phi(t,\nu^{N,i}_{\boldsymbol{x}})(\sigma^2+\sigma_0^2)+\frac{1}{2}\frac{1}{(N-1)^2}\underset{j\neq \tau\neq i}{\sum}\partial_{\nu\nu}\Phi(t,\nu^{N,i}_{\boldsymbol{x}})\sigma_0^2.
\end{aligned}
\end{equation*}

Using Theorem \ref{thm:Phi}, we obtain the boundedness of $\partial_{\nu}\Phi$ and $\partial_{\nu\nu}\Phi$. Thus,
\[ |\frac{1}{2}\frac{1}{N-1}\partial_{\nu\nu}\Phi(t,\nu^{N,i}_{\boldsymbol{x}})(\sigma^2+\sigma_0^2) |\leq \frac{C}{N-1} \]
and
\begin{equation*}
\begin{aligned}
&\frac{1}{2}\frac{1}{(N-1)^2}\underset{j\neq \tau\neq i}{\sum}\partial_{\nu\nu}\Phi(t,\nu^{N,i}_{\boldsymbol{x}})\sigma_0^2\\
=&\frac{1}{2}(N-1)(N-2)\frac{1}{(N-1)^2}\partial_{\nu\nu}\Phi(t,\nu^{N,i}_{\boldsymbol{x}})\sigma_0^2=\frac{1}{2}\partial_{\nu\nu}\Phi(t,\nu^{N,i}_{\boldsymbol{x}})\sigma_0^2+O(\frac{1}{N}). 
\end{aligned}
\end{equation*}

Based on these estimates, one can obtain 
\begin{equation}
\label{final1}
\frac{1}{2}\underset{j=1}{\overset{N}{\sum}}\partial_{x^jx^j}\phi^{N,i}(t,\boldsymbol{x})\sigma^2+\frac{1}{2}\underset{j,\tau=1}{\overset{N}{\sum}}\partial_{x^j}\partial_{x^\tau}\phi^{N,i}(t,\boldsymbol{x})\sigma_0^2
=\frac{1}{2}\partial_{\nu\nu}\Phi(t,\nu^{N,i}_{\boldsymbol{x}})\sigma_0^2+O(\frac{1}{N}).
\end{equation}

We then calculate
\begin{equation}
\begin{aligned}\label{eq:2ndnashsystem}
&\underset{j=1}{\overset{N}{\sum}}\partial_{x^j}\phi^{N,i}(t,\boldsymbol{x})\Big[(A-B^2R^{-1}P_t)x^j-B^2R^{-1}\phi^{N,j}(t,\boldsymbol{x})-Bh\Big(k_j^N(t,\nu^{N,1}_{\boldsymbol{x}},\ldots,\nu^{N,N}_{\boldsymbol{x}},\lambda^{N,1}_{\boldsymbol{\phi}(t,\boldsymbol{x})},\ldots,\lambda^{N,N}_{\boldsymbol{\phi}(t,\boldsymbol{x})})\Big)\Big]\\
=&\underset{j\neq i}{\sum}\partial_{x^j}\phi^{N,i}(t,\boldsymbol{x})\Big[(A-B^2R^{-1}P_t)x^j-B^2R^{-1}\phi^{N,j}(t,\boldsymbol{x})-Bh\Big(k_j^N(t,\nu^{N,1}_{\boldsymbol{x}},\ldots,\nu^{N,N}_{\boldsymbol{x}},\lambda^{N,1}_{\boldsymbol{\phi}(t,\boldsymbol{x})},\ldots,\lambda^{N,N}_{\boldsymbol{\phi}(t,\boldsymbol{x})})\Big)\Big]\\
=&\partial_{\nu}\Phi(t,\nu^{N,i}_{\boldsymbol{x}})\Big[(A-B^2R^{-1}P_t)\nu^{N,i}_{\boldsymbol{x}}-B^2R^{-1}\frac{1}{N-1}\underset{j\neq i}{\sum}\phi^{N,j}(t,\boldsymbol{x})\\
&-B\frac{1}{N-1}\underset{j\neq i}{\sum}h\Big(k_j^N(t,\nu^{N,1}_{\boldsymbol{x}},\ldots,\nu^{N,N}_{\boldsymbol{x}},\lambda^{N,1}_{\boldsymbol{\phi}(t,\boldsymbol{x})},\ldots,\lambda^{N,N}_{\boldsymbol{\phi}(t,\boldsymbol{x})})\Big)\Big].
\end{aligned}
\end{equation}

By the boundedness of $\partial_{\nu}\Phi$, we have 
\begin{equation*}
\begin{aligned}
&|\frac{1}{N-1}\underset{j\neq i}{\sum}\phi^{N,j}(t,\boldsymbol{x})-\Phi(t,\nu^{N,i}_{\boldsymbol{x}})|\leq|\frac{1}{N-1}\underset{j\neq i}{\sum}\Phi(t,\nu^{N,j}_{\boldsymbol{x}})-\Phi(t,\nu^{N,i}_{\boldsymbol{x}})|\\
&\leq \frac{C}{N-1}\underset{j\neq i}{\sum}|\nu^{N,j}_{\boldsymbol{x}}-\nu^{N,i}_{\boldsymbol{x}}|=\frac{C}{N-1}\Big(\frac{1}{N-1}\underset{j\neq i}{\sum}|x^i-x^j|\Big)=O\Big(\frac{1}{N^2}\underset{j\not= i}{\sum}|x^i-x^j|\Big).
\end{aligned}
\end{equation*}

Moreover, based on \eqref{kiphik} and by the uniform Lipschitz continuity of $k$ and $\Phi$, we further have 
\begin{equation*}
\begin{aligned}
&|\frac{1}{N-1}\underset{j\neq i}{\sum}k_j^N(t,\nu^{N,1}_{\boldsymbol{x}},\ldots,\nu^{N,N}_{\boldsymbol{x}},\lambda^{N,1}_{\boldsymbol{\phi}(t,\boldsymbol{x})},\ldots,\lambda^{N,N}_{\boldsymbol{\phi}(t,\boldsymbol{x})})-k\big(t,\nu^{N,i}_{\boldsymbol{x}},\Phi(t,\nu^{N,i}_{\boldsymbol{x}})\big)|\\
\leq &\frac{1}{N-1}\underset{j\neq i}{\sum}|k_j^N(t,\nu^{N,1}_{\boldsymbol{x}},\ldots,\nu^{N,N}_{\boldsymbol{x}},\lambda^{N,1}_{\boldsymbol{\phi}(t,\boldsymbol{x})},\ldots,\lambda^{N,N}_{\boldsymbol{\phi}(t,\boldsymbol{x})})-k\big(t,\nu^{N,j}_{\boldsymbol{x}},\Phi(t,\nu^{N,j}_{\boldsymbol{x}})\big)|\\
&+\frac{1}{N-1}\underset{j\neq i}{\sum}|k\big(t,\nu^{N,j}_{\boldsymbol{x}},\Phi(t,\nu^{N,j}_{\boldsymbol{x}})\big)-k\big(t,\nu^{N,i}_{\boldsymbol{x}},\Phi(t,\nu^{N,i}_{\boldsymbol{x}})\big)|\\
\leq &\frac{C}{N-1}\underset{j\neq i}{\sum}\frac{1}{N-1}\Big(1+\frac{1}{N-1}\underset{\tau\neq j}{\sum}|x^j-x^\tau|\Big)+\frac{C}{N-1}\underset{j\neq i}{\sum}|\nu^{N,j}_{\boldsymbol{x}}-\nu^{N,i}_{\boldsymbol{x}}|\\
\leq &\frac{C}{N-1}\Big(1+\frac{1}{N-1}\underset{j\neq i}{\sum}|x^i-x^j|+\frac{1}{(N-1)^2}\underset{j\neq i}{\sum}\underset{\tau\neq j}{\sum}|x^j-x^\tau|\Big).
\end{aligned}
\end{equation*}

In fact, we can show that
\begin{equation}\label{suminequality}
\frac{1}{(N-1)^2}\underset{j\neq i}{\sum}\underset{\tau\neq j}{\sum}|x^j-x^\tau|\leq \frac{2N-3}{(N-1)^2}\underset{j\neq i}{\sum}|x^i-x^j|.
\end{equation}

Thus, we have 
\begin{equation*}
|\frac{1}{N-1}\underset{j\neq i}{\sum}k_j^N(t,\nu^{N,1}_{\boldsymbol{x}},\ldots,\nu^{N,N}_{\boldsymbol{x}},\lambda^{N,1}_{\boldsymbol{\phi}(t,\boldsymbol{x})},\ldots,\lambda^{N,N}_{\boldsymbol{\phi}(t,\boldsymbol{x})})-k\big(t,\nu^{N,i}_{\boldsymbol{x}},\Phi(t,\nu^{N,i}_{\boldsymbol{x}})\big)|=O\Big( \frac{1}{N}\Big(1+\frac{1}{N}\underset{j\not=i}{\sum}|x^i-x^j|\Big)     \Big).
\end{equation*}

Combining all the above estimates,  we derive from \eqref{eq:2ndnashsystem}
\begin{equation}\label{eq:2ndterm}
\begin{aligned}
&\underset{j=1}{\overset{N}{\sum}}\partial_{x^j}\phi^{N,i}(t,\boldsymbol{x})\Big[(A-B^2R^{-1}P_t)x^j-B^2R^{-1}\phi^{N,j}(t,\boldsymbol{x})-Bh\Big(k_j^N(t,\nu^{N,1}_{\boldsymbol{x}},\ldots,\nu^{N,N}_{\boldsymbol{x}},\lambda^{N,1}_{\boldsymbol{\phi}(t,\boldsymbol{x})},\ldots,\lambda^{N,N}_{\boldsymbol{\phi}(t,\boldsymbol{x})})\Big)\Big]\\
=&\partial_{\nu}\Phi(t,\nu^{N,i}_{\boldsymbol{x}})\Big[(A-B^2R^{-1}P_t)\nu^{N,i}_{\boldsymbol{x}}-B^2R^{-1}\Phi(t,\nu^{N,i}_{\boldsymbol{x}})-Bh\Big(k\big(t,\nu^{N,i}_{\boldsymbol{x}},\Phi(t,\nu^{N,i}_{\boldsymbol{x}})\Big)\Big]\\
&+O\Big( \frac{1}{N}\Big(1+\frac{1}{N}\underset{j\not=i}{\sum}|x^i-x^j|\Big)     \Big).
\end{aligned}
\end{equation} 

At the end, we use \eqref{kiphik} to estimate
\begin{equation}\label{eq:3rdterm}
\begin{aligned}
&P_tBh\Big(k_i^N(t,\nu^{N,1}_{\boldsymbol{x}},\ldots,\nu^{N,N}_{\boldsymbol{x}},\lambda^{N,1}_{\boldsymbol{\phi}(t,\boldsymbol{x})},\ldots,\lambda^{N,N}_{\boldsymbol{\phi}(t,\boldsymbol{x})})\Big)\\
=&P_tBh\Big(k\big(t,\nu^{N,i}_{\boldsymbol{x}},\Phi(t,\nu^{N,i}_{\boldsymbol{x}})\big)\Big)+O\Big( \frac{1}{N}\Big(1+\frac{1}{N}\underset{j\not = i}{\sum}|x^i-x^j|\Big)     \Big).
\end{aligned}
\end{equation}

Combing \eqref{final1}, \eqref{eq:2ndterm} and \eqref{eq:3rdterm}, the desired equation \eqref{phi} is proved.
\end{proof}

We consider the following two SDEs on the interval $[t_0,T]$:
\begin{equation}
\left\{
\begin{aligned}\label{convergence:xi}
dx^i_t=&~\Big[(A-B^2R^{-1}P_t)x^i_t-B^2R^{-1}\phi^{N,i}(t,\boldsymbol{x}_t)-Bh\Big(k_i^N(t,\nu^{N,1}_{\boldsymbol{x}_t},\ldots,\nu^{N,N}_{\boldsymbol{x}_t},\lambda^{N,1}_{\boldsymbol{\phi}(t,\boldsymbol{x}_t)},\ldots,\lambda^{N,N}_{\boldsymbol{\phi}(t,\boldsymbol{x}_t)})\Big)\Big]dt\\
&+\sigma dW_t^i+\sigma_0dW_t^0,\\
x^i_{t_0}=&~\xi^i,
\end{aligned}
\right.
\end{equation}
and
\begin{equation}
\left\{
\begin{aligned}\label{convergence:x*i}
dx^{*,i}_t=&~\Big[(A-B^2R^{-1}P_t)x^{*,i}_t-B^2R^{-1}\Psi^{N,i}(t,\boldsymbol{x}_t^*)-Bh\Big(k_i^N(t,\nu^{N,1}_{\boldsymbol{x}_t^*},\ldots,\nu^{N,N}_{\boldsymbol{x}_t^*},\lambda^{N,1}_{\boldsymbol{\Psi}(t,\boldsymbol{x}_t^*)},\ldots,\lambda^{N,N}_{\boldsymbol{\Psi}(t,\boldsymbol{x}_t^*)})\Big)\Big]dt\\
&+\sigma dW_t^i+\sigma_0dW_t^0,\\
x^{*,i}_{t_0}=&~\xi^i.
\end{aligned}
\right.
\end{equation}

\begin{lemma}
Suppose that Assumptions (A), (B') (C) hold and $\sigma>0$. For $1\leq j \leq N$, $j\neq i$, it has 
\begin{equation}\label{xij}
\mathbb{E}[\underset{t_0\leq s\leq t}{\sup}|x_s^i-x_s^j|^2]\leq C\left(1+\mathbb{E}[|\xi^i-\xi^j|^2]\right),
\end{equation}
and
\begin{equation}\label{x*ij}
\mathbb{E}[\underset{t_0\leq s\leq t}{\sup}|x_s^{*,i}-x_s^{*,j}|^2]\leq C\left(1+\mathbb{E}[|\xi^i-\xi^j|^2]\right).
\end{equation}
\end{lemma}
\begin{proof}
From \eqref{convergence:xi}, we have for $1\leq j \leq N$, $j\neq i$,
\begin{equation*}
\begin{aligned}
&|x_t^i-x_t^j|\leq  |\xi^i-\xi^j|+C\int_{t_0}^t|x_s^i-x_s^j|ds+C\int_{t_0}^t|\phi^{N,i}(s,\boldsymbol{x}_s)-\phi^{N,j}(s,\boldsymbol{x}_s)|ds\\
&~+C\int_{t_0}^t|h\Big(k_i^N(s,\nu^{N,1}_{\boldsymbol{x}_s},\ldots,\nu^{N,N}_{\boldsymbol{x}_s},\lambda^{N,1}_{\boldsymbol{\phi}(s,\boldsymbol{x}_s)},\ldots,\lambda^{N,N}_{\boldsymbol{\phi}(s,\boldsymbol{x}_s)})\Big)-h\Big(k_j^N(s,\nu^{N,1}_{\boldsymbol{x}_s},\ldots,\nu^{N,N}_{\boldsymbol{x}_s},\lambda^{N,1}_{\boldsymbol{\phi}(s,\boldsymbol{x}_s)},\ldots,\lambda^{N,N}_{\boldsymbol{\phi}(s,\boldsymbol{x}_s)})\Big)|ds.
\end{aligned}
\end{equation*}

By the boundedness of $\phi^{N,i},1\leq i\leq N$, and $h$, then using the Gronwall's inequality, we derive \eqref{xij}. Similarly, we can prove \eqref{x*ij}. 
\end{proof}

\begin{theorem}\label{uvestimate}
Suppose that Assumptions (A), (B'), (C) hold and $\sigma>0$. Let $\boldsymbol{x}:=(x^{i})_{1\leq i\leq N}$ and $\boldsymbol{x}^{*}:=(x^{*,i})_{1\leq i\leq N}$ be solutions to \eqref{convergence:xi} and \eqref{convergence:x*i}, respectively. Then there exists a constant $C>0$, independent of $N$, satisfying 
\begin{equation}\label{est:xix*i}
\mathbb{E}\Big[\underset{t_0\leq t\leq T}{\sup}|x_t^i-x_t^{*,i}|^2\Big]\leq\frac{C}{N^2}\Big(1+\frac{1}{N}\underset{j=1}{\overset{N}\sum}\mathbb{E}[|\xi^i-\xi^j|^2]+\frac{1}{N^2}\underset{i,j=1}{\overset{N}{\sum}}\mathbb{E}[|\xi^{i}-\xi^{j}|^2]\Big),\end{equation}
\begin{equation}\label{est:phipsi}
\mathbb{E}\Big[\underset{t_0\leq t\leq T}{\sup}|\phi^{N,i}(t,\boldsymbol{x}_t^*)-\Psi^{N,i}(t,\boldsymbol{x}_t^*)|^2\Big]\leq\frac{C}{N^2}\Big(1+\frac{1}{N}\underset{j=1}{\overset{N}\sum}\mathbb{E}[|\xi^i-\xi^j|^2]+\frac{1}{N^2}\underset{i,j=1}{\overset{N}{\sum}}\mathbb{E}[|\xi^{i}-\xi^{j}|^2]\Big).
\end{equation}
Moreover, we have 
\begin{equation}\label{est:initial} 
|\phi^{N,i}(t_0,\boldsymbol{\xi})-\Psi^{N,i}(t_0,\boldsymbol{\xi})|\leq\frac{C}{N}\Big(1+\frac{1}{N}\underset{j=1}{\overset{N}\sum}|\xi^{i}-\xi^{j}|^2+\frac{1}{N^2}\underset{i,j=1}{\overset{N}{\sum}}|\xi^{i}-\xi^{j}|^2\Big)^{\frac{1}{2}}.
\end{equation}
\end{theorem}
\begin{proof} Without loss of generality, we can assume that $t_0=0$.
Applying the It\^o's formula to $\phi^{N,i}(t,\boldsymbol{x}_t^*)$ and by \eqref{phi}, it follows 
\begin{equation*}
\begin{aligned}
&d\phi^{N,i}(t,\boldsymbol{x}_t^*)\\
=&\Big\{\underset{j=1}{\overset{N}{\sum}}\partial_{x^j}\phi^{N,i}(t,\boldsymbol{x}^*_t)\Big[B^2R^{-1}\big(\phi^{N,j}(t,\boldsymbol{x}^*_t) -\Psi^{N,j}(t,\boldsymbol{x}^*_t)\big)+Bh\Big(k_j^N(t,\nu^{N,1}_{\boldsymbol{x}^*_t},\ldots,\nu^{N,N}_{\boldsymbol{x}^*_t},\lambda^{N,1}_{\boldsymbol{\phi}(t,\boldsymbol{x}_t^*)},\ldots,\lambda^{N,N}_{\boldsymbol{\phi}(t,\boldsymbol{x}_t^*)})\Big)\\
&-Bh\Big(k_j^N(t,\nu^{N,1}_{\boldsymbol{x}^*_t},\ldots,\nu^{N,N}_{\boldsymbol{x}^*_t},\lambda^{N,1}_{\boldsymbol{\Psi}(t,\boldsymbol{x}_t^*)},\ldots,\lambda^{N,N}_{\boldsymbol{\Psi}(t,\boldsymbol{x}_t^*)})\Big)\Big]-(A-B^2R^{-1}P_t)\phi^{N,i}(t,\boldsymbol{x}^*_t)-Ql(\nu^{N,i}_{\boldsymbol{x}^*_t})\\
&+P_tBh\Big(k_i^N(t,\nu^{N,1}_{\boldsymbol{x}^*_t},\ldots,\nu^{N,N}_{\boldsymbol{x}^*_t},\lambda^{N,1}_{\boldsymbol{\phi}(t,\boldsymbol{x}_t^*)},\ldots,\lambda^{N,N}_{\boldsymbol{\phi}(t,\boldsymbol{x}_t^*)})\Big)+r^{N,i}(t,\boldsymbol{x}^*_t)\Big\}dt+\underset{j=1}{\overset{N}{\sum}}\partial_{x^j}\phi^{N,i}(t,\boldsymbol{x}^*_t)\Big(\sigma dW_t^j+\sigma_0dW_t^0\Big).
\end{aligned}
\end{equation*}

Similarly, using \eqref{Psi}, we can derive $\Psi^{N,i}(t,\boldsymbol{x}_t^*)$ satisfies 
\begin{equation*}
\begin{aligned}
&d\Psi^{N,i}(t,\boldsymbol{x}_t^*)\\
=&\Big\{-(A-B^2R^{-1}P_t)\Psi^{N,i}(t,\boldsymbol{x}^*_t)-Ql(\nu^{N,i}_{\boldsymbol{x}^*_t})+P_tBh\Big(k_i^N(t,\nu^{N,1}_{\boldsymbol{x}^*_t},\ldots,\nu^{N,N}_{\boldsymbol{x}^*_t},\lambda^{N,1}_{\boldsymbol{\Psi}(t,\boldsymbol{x}_t^*)},\ldots,\lambda^{N,N}_{\boldsymbol{\Psi}(t,\boldsymbol{x}_t^*)}) \Big)\Big\}dt\\
&+\underset{j=1}{\overset{N}{\sum}}\partial_{x^j}\Psi^{N,i}(t,\boldsymbol{x}^*_t)\Big(\sigma dW_t^j+\sigma_0dW_t^0\Big).
\end{aligned}
\end{equation*}

Then, after simple calculation, we have 
\begin{equation}\label{phipsdifference}
\begin{aligned}
&d\big(\phi^{N,i}(t,\boldsymbol{x}^*_t) -\Psi^{N,i}(t,\boldsymbol{x}^*_t)\big)^2\\=&\Big\{ 2\underset{j=1}{\overset{N}{\sum}}\big(\phi^{N,i}(t,\boldsymbol{x}^*_t) -\Psi^{N,i}(t,\boldsymbol{x}^*_t)\big)\partial_{x^j}\phi^{N,i}(t,\boldsymbol{x}^*_t)\Big[B^2R^{-1}\big(\phi^{N,j}(t,\boldsymbol{x}^*_t) -\Psi^{N,j}(t,\boldsymbol{x}^*_t)\big)\\&+Bh\Big(k_j^N(t,\nu^{N,1}_{\boldsymbol{x}^*_t},\ldots,\nu^{N,N}_{\boldsymbol{x}^*_t},\lambda^{N,1}_{\boldsymbol{\phi}(t,\boldsymbol{x}_t^*)},\ldots,\lambda^{N,N}_{\boldsymbol{\phi}(t,\boldsymbol{x}_t^*)})\Big)-Bh\Big(k_j^N(t,\nu^{N,1}_{\boldsymbol{x}^*_t},\ldots,\nu^{N,N}_{\boldsymbol{x}^*_t},\lambda^{N,1}_{\boldsymbol{\Psi}(t,\boldsymbol{x}_t^*)},\ldots,\lambda^{N,N}_{\boldsymbol{\Psi}(t,\boldsymbol{x}_t^*)})\Big)\Big]\\
&+2\big(\phi^{N,i}(t,\boldsymbol{x}^*_t) -\Psi^{N,i}(t,\boldsymbol{x}^*_t)\big)P_tB\Big [h\Big(k_i^N(t,\nu^{N,1}_{\boldsymbol{x}^*_t},\ldots,\nu^{N,N}_{\boldsymbol{x}^*_t},\lambda^{N,1}_{\boldsymbol{\phi}(t,\boldsymbol{x}_t^*)},\ldots,\lambda^{N,N}_{\boldsymbol{\phi}(t,\boldsymbol{x}_t^*)})\Big)\\
&-h\Big(k_i^N(t,\nu^{N,1}_{\boldsymbol{x}^*_t},\ldots,\nu^{N,N}_{\boldsymbol{x}^*_t},\lambda^{N,1}_{\boldsymbol{\Psi}(t,\boldsymbol{x}_t^*)},\ldots,\lambda^{N,N}_{\boldsymbol{\Psi}(t,\boldsymbol{x}_t^*)}) \Big)\Big]\\
&-2(A-B^2R^{-1}P_t)\big(\phi^{N,i}(t,\boldsymbol{x}^*_t) -\Psi^{N,i}(t,\boldsymbol{x}^*_t)\big)^2+2\big(\phi^{N,i}(t,\boldsymbol{x}^*_t) -\Psi^{N,i}(t,\boldsymbol{x}^*_t)\big)r^{N,i}(t,\boldsymbol{x}^*_t)\Big\}dt\\
&+\Big[\underset{j=1}{\overset{N}{\sum}}|\sigma\big(\partial_{x^j}\phi^{N,i}(t,\boldsymbol{x}^*_t) -\partial_{x^j}\Psi^{N,i}(t,\boldsymbol{x}^*_t)\big)|^2+\Big|\underset{j=1}{\overset{N}{\sum}}\sigma_0\big(\partial_{x^j}\phi^{N,i}(t,\boldsymbol{x}^*_t) -\partial_{x^j}\Psi^{N,i}(t,\boldsymbol{x}^*_t)\big)\Big|^2\Big]dt\\
&+2\underset{j=1}{\overset{N}{\sum}}\big(\phi^{N,i}(t,\boldsymbol{x}^*_t) -\Psi^{N,i}(t,\boldsymbol{x}^*_t)\big)\big(\partial_{x^j}\phi^{N,i}(t,\boldsymbol{x}^*_t) -\partial_{x^j}\Psi^{N,i}(t,\boldsymbol{x}^*_t)\big)\Big(\sigma dW_t^j+\sigma_0dW_t^0\Big).
\end{aligned}
\end{equation}

Recall that $\partial_{x^i}\phi^{N,i}(t,\boldsymbol{x})=0$ and $\partial_{x^j}\phi^{N,i}(t,\boldsymbol{x})$, $j\neq i$ is bounded by $\frac{C}{N}$. Integrating \eqref{phipsdifference} from $t$ to $T$ and then taking the conditional expectation on $\xi$, one can get
{\small\begin{equation}
\begin{aligned}\label{difference:phipsi}
&\mathbb{E}^{\boldsymbol{\xi}}[|\phi^{N,i}(t,\boldsymbol{x}^*_t) -\Psi^{N,i}(t,\boldsymbol{x}^*_t)|^2]\\\leq &\mathbb{E}^{\boldsymbol{\xi}}[|\phi^{N,i}(T,\boldsymbol{x}^*_T) -\Psi^{N,i}(T,\boldsymbol{x}^*_T)|^2]+C\mathbb{E}^{\boldsymbol{\xi}}\Big[\int_t^T|\phi^{N,i}(s,\boldsymbol{x}^*_s) -\Psi^{N,i}(s,\boldsymbol{x}^*_s)|^2ds\Big]\\
&+\frac{C}{N}\underset{j\neq i}{\sum}\mathbb{E}^{\boldsymbol{\xi}}\Big[\int_t^T|\phi^{N,j}(s,\boldsymbol{x}^*_s) -\Psi^{N,j}(s,\boldsymbol{x}^*_s)|^2ds\Big]+C\mathbb{E}^{\boldsymbol{\xi}}\Big[\int_t^T|r^{N,i}(s,\boldsymbol{x}^*_s)|^2ds\Big]\\
&+C\mathbb{E}^{\boldsymbol{\xi}}\Big[\int_t^T|k_i^N(s,\nu^{N,1}_{\boldsymbol{x}^*_s},\ldots,\nu^{N,N}_{\boldsymbol{x}^*_s},\lambda^{N,1}_{\boldsymbol{\phi}(s,\boldsymbol{x}_s^*)},\ldots,\lambda^{N,N}_{\boldsymbol{\phi}(s,\boldsymbol{x}_s^*)})-k_i^N(s,\nu^{N,1}_{\boldsymbol{x}^*_s},\ldots,\nu^{N,N}_{\boldsymbol{x}^*_s},\lambda^{N,1}_{\boldsymbol{\Psi}(s,\boldsymbol{x}_s^*)},\ldots,\lambda^{N,N}_{\boldsymbol{\Psi}(s,\boldsymbol{x}_s^*)}) |^2ds\Big]\\
&+\frac{C}{N}\underset{j\neq i}{\sum}\mathbb{E}^{\boldsymbol{\xi}}\Big[\int_t^T|k_j^N(s,\nu^{N,1}_{\boldsymbol{x}^*_s},\ldots,\nu^{N,N}_{\boldsymbol{x}^*_s},\lambda^{N,1}_{\boldsymbol{\phi}(s,\boldsymbol{x}_s^*)},\ldots,\lambda^{N,N}_{\boldsymbol{\phi}(s,\boldsymbol{x}_s^*)})-k_j^N(s,\nu^{N,1}_{\boldsymbol{x}^*_s},\ldots,\nu^{N,N}_{\boldsymbol{x}^*_s},\lambda^{N,1}_{\boldsymbol{\Psi}(s,\boldsymbol{x}_s^*)},\ldots,\lambda^{N,N}_{\boldsymbol{\Psi}(s,\boldsymbol{x}_s^*)}) |^2ds\Big]. 
\end{aligned}
\end{equation}}

According to \eqref{rni} and \eqref{x*ij}, we have 
\begin{equation}
\begin{aligned}\label{estimate:rni}
\mathbb{E}^{\boldsymbol{\xi}}[|r^{N,i}(t,\boldsymbol{x}^*_t)|^2]\leq& \frac{C}{N^2}\Big(1+\frac{1}{N}\underset{j=1}{\overset{N}\sum}\mathbb{E}^{\boldsymbol{\xi}}[|x^{*,i}_t-x^{*,j}_t|^2]\Big)\\
\leq & \frac{C}{N^2}\Big(1+\frac{1}{N}\underset{j=1}{\overset{N}\sum}|\xi^{i}-\xi^{j}|^2]\Big).
\end{aligned}
\end{equation}

Using \eqref{kiphikpsi},  \eqref{suminequality} and \eqref{x*ij}, we can get
{\small\begin{equation}
\begin{aligned}\label{kiphipsi}
&\mathbb{E}^{\boldsymbol{\xi}}\Big[\int_t^T|k_i^N(s,\nu^{N,1}_{\boldsymbol{x}^*_s},\ldots,\nu^{N,N}_{\boldsymbol{x}^*_s},\lambda^{N,1}_{\boldsymbol{\phi}(s,\boldsymbol{x}^*_s)},\ldots,\lambda^{N,N}_{\boldsymbol{\phi}(s,\boldsymbol{x}^*_s)})-k_i^N(s,\nu^{N,1}_{\boldsymbol{x}^*_s},\ldots,\nu^{N,N}_{\boldsymbol{x}^*_s},\lambda^{N,1}_{\boldsymbol{\Psi}(s,\boldsymbol{x}^*_s)},\ldots,\lambda^{N,N}_{\boldsymbol{\Psi}(s,\boldsymbol{x}^*_s)}) |^2ds\Big]\\
&+\frac{1}{N}\underset{j\neq i}{\sum}\mathbb{E}^{\boldsymbol{\xi}}\Big[\int_t^T|k_j^N(s,\nu^{N,1}_{\boldsymbol{x}^*_s},\ldots,\nu^{N,N}_{\boldsymbol{x}^*_s},\lambda^{N,1}_{\boldsymbol{\phi}(s,\boldsymbol{x}^*_s)},\ldots,\lambda^{N,N}_{\boldsymbol{\phi}(s,\boldsymbol{x}^*_s)})-k_j^N(s,\nu^{N,1}_{\boldsymbol{x}^*_s},\ldots,\nu^{N,N}_{\boldsymbol{x}^*_s},\lambda^{N,1}_{\boldsymbol{\Psi}(s,\boldsymbol{x}^*_s)},\ldots,\lambda^{N,N}_{\boldsymbol{\Psi}(s,\boldsymbol{x}^*_s)}) |^2ds\Big]\\
\leq & \frac{C}{N^2}\Big(1+\frac{1}{N}\underset{j=1}{\overset{N}\sum}|\xi^i-\xi^j|^2\Big)+\frac{C}{N}\underset{j\neq i}{\sum}\mathbb{E}^{\boldsymbol{\xi}}\Big[\int_t^T|\phi^{N,j}(s,\boldsymbol{x}^*_s) -\Psi^{N,j}(s,\boldsymbol{x}^*_s)|^2ds\Big].
\end{aligned}
\end{equation}}
Collecting \eqref{difference:phipsi}-\eqref{kiphipsi} and observing $\phi^{N,i}(T,\boldsymbol{x}^*_T) -\Psi^{N,i}(T,\boldsymbol{x}^*_T)=0$, we deduce 
\begin{equation}\label{eq:phiNiPsiNixixix}
\begin{aligned}
\mathbb{E}^{\boldsymbol{\xi}}[|\phi^{N,i}(t,\boldsymbol{x}^*_t) -\Psi^{N,i}(t,\boldsymbol{x}^*_t)|^2]
\leq &\frac{C}{N^2}\Big(1+\frac{1}{N}\underset{j=1}{\overset{N}\sum}|\xi^{i}-\xi^{j}|^2]\Big)+C\mathbb{E}^{\boldsymbol{\xi}}\Big[\int_t^T|\phi^{N,i}(s,\boldsymbol{x}^*_s) -\Psi^{N,i}(s,\boldsymbol{x}^*_s)|^2ds\Big]\\
 &+\frac{C}{N}\underset{j\neq i}{\sum}\mathbb{E}^{\boldsymbol{\xi}}\Big[\int_t^T|\phi^{N,j}(s,\boldsymbol{x}^*_s) -\Psi^{N,j}(s,\boldsymbol{x}^*_s)|^2ds\Big].
\end{aligned}
\end{equation}
Taking the mean over the index $i\in\{1,\ldots,N \}$ and by the Gronwall's inequality, it follows
\begin{equation}\label{sumphiPsi}
\underset{0\leq t\leq T}{\sup}\Bigg[\frac{1}{N}\underset{i=1}{\overset{N}{\sum}}\mathbb{E}^{\boldsymbol{\xi}}[|\phi^{N,i}(t,\boldsymbol{x}^*_t) -\Psi^{N,i}(t,\boldsymbol{x}^*_t)|^2]
      \Bigg]\leq \frac{C}{N^2}\Big(1+\frac{1}{N^2}\underset{i,j=1}{\overset{N}{\sum}}|\xi^{i}-\xi^{j}|^2\Big)
\end{equation} 
Substituting \eqref{sumphiPsi} into \eqref{eq:phiNiPsiNixixix} and using the Gronwall's inequality again, we finally deduce 
\begin{equation*}
\underset{0\leq t\leq T}{\sup}\mathbb{E}^{\boldsymbol{\xi}}[|\phi^{N,i}(t,\boldsymbol{x}^*_t) -\Psi^{N,i}(t,\boldsymbol{x}^*_t)|^2]
\leq \frac{C}{N^2}\Big(1+\frac{1}{N}\underset{j=1}{\overset{N}\sum}|\xi^{i}-\xi^{j}|^2+\frac{1}{N^2}\underset{i,j=1}{\overset{N}{\sum}}|\xi^{i}-\xi^{j}|^2\Big).
\end{equation*}
Noting that $\phi^{N,i}(0,\boldsymbol{x}^*_0) -\Psi^{N,i}(0,\boldsymbol{x}^*_0)=\phi^{N,i}(0,\boldsymbol{\xi}) -\Psi^{N,i}(0,\boldsymbol{\xi})$, thus the following holds for $1\leq i\leq N$ $\mathbb{P}$-a.s. 
\begin{equation}\label{eq:phiNiPsiNit00}
|\phi^{N,i}(0,\boldsymbol{\xi})-\Psi^{N,i}(0,\boldsymbol{\xi})|\leq\frac{C}{N}\Big(1+\frac{1}{N}\underset{j=1}{\overset{N}\sum}|\xi^{i}-\xi^{j}|^2+\frac{1}{N^2}\underset{i,j=1}{\overset{N}{\sum}}|\xi^{i}-\xi^{j}|^2\Big)^{\frac{1}{2}},
\end{equation}
which is exactly \eqref{est:initial} when $t_0=0$. In fact, \eqref{eq:phiNiPsiNit00} implies that, for any $t\in[0,T]$,
\begin{equation*}
|\phi^{N,i}(t,\boldsymbol{x}^*_t) -\Psi^{N,i}(t,\boldsymbol{x}^*_t)|^2
\leq \frac{C}{N^2}\Big(1+\frac{1}{N}\underset{j=1}{\overset{N}\sum}|x_t^{*,i}-x_t^{*,j}|^2+\frac{1}{N^2}\underset{i,j=1}{\overset{N}{\sum}}|x_t^{*,i}-x_t^{*,j}|^2\Big).
\end{equation*}
Thus,
\begin{equation*}
\mathbb E\Big[\sup_{0\leq t\leq T}|\phi^{N,i}(t,\boldsymbol{x}^*_t) -\Psi^{N,i}(t,\boldsymbol{x}^*_t)|^2\Big]
\leq \frac{C}{N^2}\Big(1+\frac{1}{N}\underset{j=1}{\overset{N}\sum}\mathbb E\big[|\xi^{i}-\xi^{j}|^2\big]+\frac{1}{N^2}\underset{i,j=1}{\overset{N}{\sum}}\mathbb E\big[|\xi^{i}-\xi^{j}|^2\big]\Big),
\end{equation*}
which is \eqref{est:phipsi}.

In the remaining part, we aims to discuss the difference between $x_t^i$ and $x_t^{*,i}$ for any $t\in [0,T]$ and $1\leq i \leq N$. 
From \eqref{convergence:xi} and \eqref{convergence:x*i}, it follows 
\begin{equation}
\begin{aligned}\label{dif:xix*i}
&|x_t^i-x_t^{*,i}|\\\leq& C\int_0^t|x_s^i-x_s^{*,i}|ds+C\int_0^t|\phi^{N,i}(s,\boldsymbol{x}_s) -\Psi^{N,i}(s,\boldsymbol{x}^*_s)|ds\\
&+C\int_0^t|k_i^N(s,\nu^{N,1}_{\boldsymbol{x}_s},\ldots,\nu^{N,N}_{\boldsymbol{x}_s},\lambda^{N,1}_{\boldsymbol{\phi}(s,\boldsymbol{x}_s)},\ldots,\lambda^{N,N}_{\boldsymbol{\phi}(s,\boldsymbol{x}_s)})-k_i^N(s,\nu^{N,1}_{\boldsymbol{x}^*_s},\ldots,\nu^{N,N}_{\boldsymbol{x}^*_s},\lambda^{N,1}_{\boldsymbol{\Psi}(s,\boldsymbol{x}^*_s)},\ldots,\lambda^{N,N}_{\boldsymbol{\Psi}(s,\boldsymbol{x}^*_s)}) |ds.
\end{aligned}
\end{equation}
Firstly, recalling $\phi^{N,i}(t,\boldsymbol{x})=\Phi(t,\nu^{N,i}_{\boldsymbol{x}})$, we have 
\begin{equation}\label{phixpsix*}
\begin{aligned}
|\phi^{N,i}(t,\boldsymbol{x}_t) -\Psi^{N,i}(t,\boldsymbol{x}^*_t)|\leq& |\Phi(t,\nu^{N,i}_{\boldsymbol{x}_t}) -\Phi(t,\nu^{N,i}_{\boldsymbol{x}_t^*})|+|\phi^{N,i}(t,\boldsymbol{x}_t^*)-\Psi^{N,i}(t,\boldsymbol{x}_t^*)|\\
\leq &\frac{1}{N-1}\underset{j\neq i}{\sum}|x_t^j-x_t^{*,j}|+|\phi^{N,i}(t,\boldsymbol{x}_t^*)-\Psi^{N,i}(t,\boldsymbol{x}_t^*)|.
\end{aligned}
\end{equation}
Using \eqref{kinphixpsixprime}, it follows
 \begin{equation}
 \begin{aligned}\label{kiphikipsi*}
 &|k_i^N(t,\nu^{N,1}_{\boldsymbol{x}_t},\ldots,\nu^{N,N}_{\boldsymbol{x}_t},\lambda^{N,1}_{\boldsymbol{\phi}(t,\boldsymbol{x}_t)},\ldots,\lambda^{N,N}_{\boldsymbol{\phi}(t,\boldsymbol{x}_t)})-k_i^N(t,\nu^{N,1}_{\boldsymbol{x}^*_t},\ldots,\nu^{N,N}_{\boldsymbol{x}^*_t},\lambda^{N,1}_{\boldsymbol{\Psi}(t,\boldsymbol{x}^*_t)},\ldots,\lambda^{N,N}_{\boldsymbol{\Psi}(t,\boldsymbol{x}^*_t)}) |\\
 \leq &\frac{C}{N}\Big(1+\frac{1}{N}\underset{j\neq i}{\sum}|x^i_t-x^j_t|+\frac{1}{N}\underset{j\neq i}{\sum}|x^{*,i}_t-x^{*,j}_t|\Big)+\frac{C}{N}\underset{j\neq i}{\sum}|x^j_t-x^{*,j}_t|+\frac{C}{N}\underset{j\neq i}{\sum} |\phi^j(t,\boldsymbol{x}_t) -\Psi^j(t,\boldsymbol{x}^*_t)|.
\end{aligned} 
 \end{equation}
 Plugging \eqref{kiphikipsi*} into \eqref{dif:xix*i} and using \eqref{est:phipsi} and \eqref{sumphiPsi}, we have 
 \begin{equation}
\begin{aligned}\label{dif:xix*iaaa}
&\mathbb E\Big[\sup_{0\leq s\leq t}|x_s^i-x_s^{*,i}|^2\Big]\\
\leq& C\int_0^t\mathbb E\Big[\sup_{0\leq \kappa\leq s}|x_ \kappa^i-x_ \kappa^{*,i}|^2\Big]ds+C\int_0^t\mathbb E\Big[\sup_{0\leq  \kappa\leq s}|\phi^{N,i}( \kappa,\boldsymbol{x}_ \kappa) -\Psi^{N,i}( \kappa,\boldsymbol{x}^*_ \kappa)|^2\Big]ds\\
&+\frac{C}{N^2}\int_0^t\Big[1+\frac{1}{N}\underset{j\neq i}{\sum}
\mathbb E\Big[\sup_{0\leq  \kappa\leq s}|x_ \kappa^i-x_ \kappa^{j}|^2\Big]+\frac{C}{N}\underset{j\neq i}{\sum}
\mathbb E\Big[\sup_{0\leq  \kappa\leq s}|x_ \kappa^{*,i}-x_ \kappa^{*,j}|^2\Big]\Big]ds\\
&+\frac{C}{N}\underset{j\neq i}{\sum}\int_0^t\mathbb E\Big[\sup_{0\leq \kappa\leq s}|x_\kappa^{j}-x_\kappa^{*,j}|^2\Big]\Big]ds+\frac{C}{N}\underset{j\neq i}{\sum}\int_0^t\mathbb E\Big[\sup_{0\leq  \kappa\leq s}|\phi^{N,j}( \kappa,\boldsymbol{x}_ \kappa) -\Psi^{N,j}( \kappa,\boldsymbol{x}^*_ \kappa)|^2\Big]ds\\
\leq&\frac{C}{N^2}\Big(1+\frac{1}{N}\sum_{j=1}^N\mathbb E\big[|\xi^i-\xi^j|^2\big]+\frac{1}{N^2}\sum_{i,j=1}^N\mathbb E\big[|\xi^i-\xi^j|^2\big]\Big)\\&+C\int_0^t\mathbb E\Big[\sup_{0\leq \kappa\leq s}|x_\kappa^i-x_\kappa^{*,i}|^2\Big]ds+\frac{C}{N}\underset{j\neq i}{\sum}\int_0^t\mathbb E\Big[\sup_{0\leq \kappa\leq s}|x_\kappa^j-x_\kappa^{*,j}|^2\Big]ds,
\end{aligned}
\end{equation}
where the last inequality is also due to \eqref{xij} and \eqref{x*ij}.

Taking the mean over the index $i\in\{1,\ldots,N \}$ and by the Gronwall's inequality, it follows
\begin{equation}\label{sumphiPsiaaaa}
\frac{1}{N}\underset{i=1}{\overset{N}{\sum}}\mathbb{E}\Big[\sup_{0\leq s\leq t}|x_s^i-x_s^{*,i}|^2\Big]
      \leq\frac{C}{N^2}\Big(1+\frac{1}{N^2}\sum_{i,j=1}^N\mathbb E\big[|\xi^i-\xi^j|^2\big]\Big).
\end{equation} 
Substituting \eqref{sumphiPsiaaaa} into \eqref{dif:xix*iaaa} and then using the Gronwall's inequality, we finally have \eqref{est:xix*i}.
\end{proof}

\subsection{Propagation of Chaos}

At the end, we consider the propagation of chaos property for the corresponding optimal trajectories. Let $\bar{x}_t^i$ be the solution to the following SDE:
\begin{equation}
\left\{
\begin{aligned}\label{barxi}
d\bar{x}^i_t=&~\Big[(A-B^2R^{-1}P_t)\bar{x}^i_t-B^2R^{-1}\Phi(t,\bar{\nu}_t^i)-Bh\Big(k\big(t,\bar{\nu}_t^i,\Phi(t,\bar{\nu}_t^i)\big)\Big)\Big]dt\\
&+\sigma dW_t^i+\sigma_0dW_t^0,\\
\bar{x}^i_{t_0}=&~\xi^i,
\end{aligned}
\right.
\end{equation}
where $\bar{\nu}_t^i=\mathbb{E}[\bar{x}^i_t|\mathcal{F}_t^{W^0}]$.

\begin{theorem}
Suppose that Assumptions (A), (B'), (C) hold and $\sigma>0$. Then we have
\begin{equation}\label{est:x*ibarxi}
\mathbb{E}\Big[\underset{t_0\leq t\leq T}{\sup}|x_t^{*,i}-\bar{x}_t^i|\Big]\leq \frac{C}{\sqrt N}\quad\text{for all $1\leq i\leq N$.}
\end{equation}
\end{theorem}

\begin{proof}
We first estimate the difference $x_t^{i}-\bar{x}_t^i$ which satisfies
\begin{equation}
\begin{aligned}\label{xibarxi}
|x_t^{i}-\bar{x}_t^i|\leq& C\int_{t_0}^t|x_s^{i}-\bar{x}_s^i|ds+C\int_{t_0}^t|\phi^{N,i}(s,\boldsymbol{x}_s)-\Phi(s,\bar{\nu}_s^i)|ds\\
&+C\int_{t_0}^t|k_i^N(s,\nu^{N,1}_{\boldsymbol{x}_s},\ldots,\nu^{N,N}_{\boldsymbol{x}_s},\lambda^{N,1}_{\boldsymbol{\phi}(s,\boldsymbol{x}_s)},\ldots,\lambda^{N,N}_{\boldsymbol{\phi}(s,\boldsymbol{x}_s)})-k\big(s,\bar{\nu}_s^i,\Phi(s,\bar{\nu}_s^i)\big)|ds. 
\end{aligned}
\end{equation}
Note that 
\begin{equation}\label{phipsinu}
|\phi^{N,i}(t,\boldsymbol{x}_t)-\Phi(t,\bar{\nu}_t^i)|=|\Phi(t,\nu^{N,i}_{\boldsymbol{x}_t})-\Phi(t,\bar{\nu}_t^i)|\leq C|\nu^{N,i}_{\boldsymbol{x}_t}-\bar{\nu}_t^i|.\end{equation}
Recalling \eqref{kiphik} and using the Lipschitz continuity of $k$, we have 
\begin{equation}
\begin{aligned}\label{kiknu}
&|k_i^N(t,\nu^{N,1}_{\boldsymbol{x}_t},\ldots,\nu^{N,N}_{\boldsymbol{x}_t},\lambda^{N,1}_{\boldsymbol{\phi}(t,\boldsymbol{x}_t)},\ldots,\lambda^{N,N}_{\boldsymbol{\phi}(t,\boldsymbol{x}_t)})-k\big(t,\bar{\nu}_t^i,\Phi(t,\bar{\nu}_t^i)\big)|\\
\leq &|k_i^N(t,\nu^{N,1}_{\boldsymbol{x}_t},\ldots,\nu^{N,N}_{\boldsymbol{x}_t},\lambda^{N,1}_{\boldsymbol{\phi}(t,\boldsymbol{x}_t)},\ldots,\lambda^{N,N}_{\boldsymbol{\phi}(t,\boldsymbol{x}_t)})-k\big(t,\nu^{N,i}_{\boldsymbol{x}_t},\Phi(t,\nu^{N,i}_{\boldsymbol{x}_t})\big)|\\
&+|k\big(t,\nu^{N,i}_{\boldsymbol{x}_t},\Phi(t,\nu^{N,i}_{\boldsymbol{x}_t})\big)-k\big(t,\bar{\nu}_t^i,\Phi(t,\bar{\nu}_t^i)\big)|\\
\leq & \frac{C}{N}\Big(1+\frac{1}{N}\underset{j=1}{\overset{N}\sum}|x_t^i-x_t^j|\Big)+C|\nu^{N,i}_{\boldsymbol{x}_t}-\bar{\nu}_t^i|. 
\end{aligned}
\end{equation}
Moreover, 
\begin{equation}\label{sumnui}
|\nu^{N,i}_{\boldsymbol{x}_t}-\bar{\nu}_t^i|\leq |\nu^{N,i}_{\boldsymbol{x}_t}-\nu^{N,i}_{\boldsymbol{\bar{x}_t}}|+|\nu^{N,i}_{\boldsymbol{\bar{x}_t}}-\bar{\nu}_t^i|\leq \frac{1}{N-1}\underset{j\neq i}{\sum}|x_t^{j}-\bar{x}_t^j|+|\nu^{N,i}_{\boldsymbol{\bar{x}_t}}-\bar{\nu}_t^i|.
\end{equation}

Now, we are devoted to considering the difference $\nu^{N,i}_{\boldsymbol{\bar{x}_t}}-\bar{\nu}_t^i$. Note that $\nu^{N,i}_{\boldsymbol{\bar{x}_t}}=\frac{1}{N-1}\underset{j\neq i}{\sum}\bar{x}_t^j$ and $\bar{\nu}_t^i=\mathbb{E}[\bar{x}^i_t|\mathcal{F}_t^{W^0}]$. By \eqref{barxi}, we have 
\begin{equation*}
\begin{aligned}
d(\nu^{N,i}_{\boldsymbol{\bar{x}_t}}-\bar{\nu}_t^i)=&\Big\{(A-B^2R^{-1}P_t)(\nu^{N,i}_{\boldsymbol{\bar{x}_t}}-\bar{\nu}_t^i)-B^2R^{-1}\Big(\frac{1}{N-1}\underset{j\neq i}{\sum}\Phi(t,\bar{\nu}_t^j)-\Phi(t,\bar{\nu}_t^i)\Big)\\
&-B\Big[\frac{1}{N-1}\underset{j\neq i}{\sum}h\Big(k\big(t,\bar{\nu}_t^j,\Phi(t,\bar{\nu}_t^i)\big)\Big)-h\Big(k\big(t,\bar{\nu}_t^i,\Phi(t,\bar{\nu}_t^i)\big)\Big)\Big]\Big\}dt+\frac{1}{N-1}\underset{j\neq i}{\sum}\sigma dW_t^j.
\end{aligned}
\end{equation*}
Using the Lipschitz continuity of $\Phi$ and $h$, we obtain 
\begin{equation*}
\begin{aligned}
&|\frac{1}{N-1}\underset{j\neq i}{\sum}\Phi(t,\bar{\nu}_t^j)-\Phi(t,\bar{\nu}_t^i)|+|\frac{1}{N-1}\underset{j\neq i}{\sum}h\Big(k\big(t,\bar{\nu}_t^j,\Phi(t,\bar{\nu}_t^i)\big)\Big)-h\Big(k\big(t,\bar{\nu}_t^i,\Phi(t,\bar{\nu}_t^i)\big)\Big)|\\
\leq & \frac{C}{N-1}\underset{j\neq i}{\sum}|\bar{\nu}_t^j-\bar{\nu}_t^i|\leq\frac{C}{N-1}\underset{j\neq i}{\sum}|\nu^{N,j}_{\boldsymbol{\bar{x}_t}}-\bar{\nu}_t^j|+\frac{C}{N-1}\underset{j\neq i}{\sum}|\nu^{N,j}_{\boldsymbol{\bar{x}_t}}-\nu^{N,i}_{\boldsymbol{\bar{x}_t}}|+C|\nu^{N,i}_{\boldsymbol{\bar{x}_t}}-\bar{\nu}_t^i|\\
\leq & \frac{C}{(N-1)^2}\underset{j\neq i}{\sum}|\bar{x}_t^i-\bar{x}_t^j|+C|\nu^{N,i}_{\boldsymbol{\bar{x}_t}}-\bar{\nu}_t^i|+\frac{C}{N-1}\underset{j\neq i}{\sum}|\nu^{N,j}_{\boldsymbol{\bar{x}_t}}-\bar{\nu}_t^j|.
\end{aligned}
\end{equation*}
By the Burkholder-Davis-Gundy inequality, it follows 
\begin{equation*}
\begin{aligned}
\mathbb{E}[\underset{t_0\leq s \leq t}{\sup}|\nu^{N,i}_{\boldsymbol{\bar{x}_s}}-\bar{\nu}_s^i|^2]\leq & C\mathbb{E}\Big[
|\frac{1}{N-1}\underset{j\neq i}{\sum}\xi^j-\mathbb{E}[\xi^i]|^2\Big]+\frac{C}{(N-1)^3}\underset{j\neq i}{\sum}\mathbb{E}\Big[\int_{t_0}^t|\bar{x}_s^i-\bar{x}_s^j|^2ds\Big]\\&+C\mathbb{E}\Big[\int_{t_0}^t|\nu^{N,i}_{\boldsymbol{\bar{x}_s}}-\bar{\nu}_s^i|^2ds\Big]+\frac{C}{N-1}\underset{j\neq i}{\sum}\mathbb{E}\Big[\int_{t_0}^t|\nu^{N,j}_{\boldsymbol{\bar{x}_s}}-\bar{\nu}_s^j|^2ds\Big]+\frac{C}{N-1}\mathbb{E}\Big[\int_{t_0}^t\sigma^2ds\Big].
\end{aligned}
\end{equation*}
Note that $\xi^i$, $1\leq i\leq N$, are i.i.d random variables. Then 
\begin{equation*}
 \mathbb{E}\Big[|\frac{1}{N-1}\underset{j\neq i}{\sum}\xi^j-\mathbb{E}[\xi^i]|^2\Big]=\frac{1}{(N-1)^2}\underset{j\neq i}{\sum}\mathbb{E}[|\xi^j-\mathbb{E}[\xi^i]|^2]+\frac{1}{(N-1)^2}\underset{j\neq k\neq i}{\sum}\mathbb{E}[(\xi^j-\mathbb{E}[\xi^i])(\xi^k-\mathbb{E}[\xi^i])].
\end{equation*}
Since $\mathbb{E}[|\xi^j-\mathbb{E}[\xi^i]|^2]$ is uniformly bounded, uniformly in $N$,  and
\begin{equation*}
\mathbb{E}[(\xi^j-\mathbb{E}[\xi^i])(\xi^k-\mathbb{E}[\xi^i])]=0,
\end{equation*}
we can deduce that 
\begin{equation*}
 \mathbb{E}[|\frac{1}{N-1}\underset{j\neq i}{\sum}\xi^j-\mathbb{E}[\xi^i]|^2]=O(\frac{1}{N}).
\end{equation*}
Using that $\mathbb{E}[|\bar{x}_s^i-\bar{x}_s^j|^2]$ is uniformly bounded, uniformly in $N$, and by the Gronwall's inequality we derive 
\begin{equation}\label{est:nui}
\mathbb{E}[\underset{t_0\leq s \leq t}{\sup}|\nu^{N,i}_{\boldsymbol{\bar{x}_s}}-\bar{\nu}_s^i|^2]\leq \frac{C}{N}.
\end{equation}
From \eqref{xibarxi}-\eqref{kiknu}, we conclude 
\begin{equation*}
\begin{aligned}
\mathbb{E}[\underset{t_0\leq s \leq t}{\sup}|x_s^{i}-\bar{x}_s^i|^2]\leq &C\mathbb{E}\Big[\int_{t_0}^t|x_s^{i}-\bar{x}_s^i|^2ds\Big]+\frac{C}{N-1}\underset{j\neq i}{\sum}\mathbb{E}\Big[\int_{t_0}^t|x_s^{j}-\bar{x}_s^j|^2ds\Big]\\
&+\frac{C}{N^2}\Big(1+\frac{1}{N}\underset{j=1}{\overset{N}\sum}\mathbb{E}\Big[\int_{t_0}^t|x^i_s-x^j_s|^2ds\Big]\Big)+C\mathbb{E}\Big[\int_{t_0}^t|\nu^{N,i}_{\boldsymbol{\bar{x}_s}}-\bar{\nu}_s^i|^2ds\Big]
\end{aligned}
\end{equation*}
Applying  \eqref{est:nui}, the uniform boundedness of $\mathbb{E}[|x_s^i|^2]$ and the Gronwall's inequality, we get
\begin{equation}\label{est:xibarxi}
\mathbb{E}\Big[\underset{t_0\leq s \leq t}{\sup}|x_s^{i}-\bar{x}_s^i|^2\Big]\leq\frac{C}{N},\quad \text{ for any } t\in[t_0,T].
\end{equation}
By \eqref{est:xix*i} and \eqref{est:xibarxi}, one can get 
\begin{equation*}
\begin{aligned}
\mathbb{E}[\underset{t_0\leq t\leq T}{\sup}|x_t^{*,i}-\bar{x}_t^i|]\leq &\mathbb{E}[\underset{t_0\leq t\leq T}{\sup}|x_t^{*,i}-x_t^i|]+\mathbb{E}[\underset{t_0\leq t\leq T}{\sup}|x_t^{i}-\bar{x}_t^i|]\\
\leq &C(\frac{1}{N}+\frac{1}{\sqrt N}),
\end{aligned}
\end{equation*}
which implies \eqref{est:x*ibarxi}.
\end{proof}

\section{Conclusions}
This paper presents a class of LQ mean field games of controls with common noises. The global well-posedness of the NCE system and the master equation is established without any monotonicity condition. We also solve $N$-player games of controls to obtain the open-loop Nash equilibria. The convergence analysis results from the $N$-player games to the mean field game and the propagation of chaos property are discussed in the end. In future work, we plan to explore the closed-loop Nash equilibria and the well-posedness of the Nash system in the $N$-player games with quadratic growth data. Moreover, we will discuss mean field games in more general settings, such as with time delays, regime-switching, Poisson jumps, etc. 

\begin{acknowledgement}
M. Li thanks the support from the China Scholarship Council (No. 202006220189). C. Mou gratefully acknowledges the support by CityU Start-up Grant 7200684 and Hong Kong RGC Grant ECS 9048215. Z. Wu thanks the support from the National Natural Science Foundation of China (No. 11831010, 61961160732), the Natural Science Foundation of Shandong Province (No. ZR2019ZD42), the Taishan Scholars Climbing Program of Shandong (No. TSPD20210302). C. Zhou is supported by Singapore MOE (Ministry of Education)
AcRF Grants R-146-000-271-112 and R-146-000-284-114, and NSFC Grant No. 11871364. 
\end{acknowledgement}

\end{document}